\documentclass[a4paper]{amsart}

\usepackage[T1]{fontenc}
\usepackage[utf8]{inputenc}
\usepackage{hyperref}

\newtheorem{theorem}{Theorem}[section]
\newtheorem{proposition}[theorem]{Proposition}
\newtheorem{lemma}[theorem]{Lemma}
\newtheorem{corollary}[theorem]{Corollary}
\newtheorem{observation}[theorem]{Observation}

\theoremstyle{definition}
\newtheorem{definition}[theorem]{Definition}
\newtheorem{notation}[theorem]{Notation}

\theoremstyle{remark}
\newtheorem{remark}[theorem]{Remark}

\numberwithin{equation}{section}

\newcommand{\Lip}{\operatorname{Lip}}
\newcommand{\Lipo}{\operatorname{Lip}_1(X,m)}
\newcommand{\dom}{\operatorname{Dom}}

\newcommand{\diam}{\operatorname{diam}}
\newcommand{\id}{\operatorname{Id}}
\newcommand{\graph}{\operatorname{Graph}}
\newcommand{\sspan}{\operatorname{span}}

\newcommand{\eps}{\varepsilon}
\renewcommand{\phi}{\varphi}
\newcommand{\Fogt}{(F\circ\gamma)' (t)}

\newcommand{\Hd}{\mathcal{H}^d}
\newcommand{\Hn}{\mathcal{H}^n}
\newcommand{\Hs}{\mathcal{H}^s}
\newcommand{\Ho}{\mathcal{H}^1}
\newcommand{\N}{\mathbb{N}}

\newcommand{\R}{\mathbb{R}}
\newcommand{\B}{\mathbb{B}}
\newcommand{\p}{\sigma}

\renewcommand{\P}{\mathbb{P}}

\begin{document}

\begin{abstract}
We characterise purely $n$-unrectifiable subsets $S$ of a complete metric space $X$ with finite Hausdorff $n$-measure by studying arbitrarily small perturbations of elements of $\Lipo$, the set of all bounded 1-Lipschitz functions $f\colon X \to \R^m$ with respect to the supremum norm.
In one such characterisation it is shown that, if $S$ has positive lower density almost everywhere, then the set of all $f$ with $\Hn(f(S))=0$ is residual in $\Lipo$.
Conversely, if $E\subset X$ is $n$-rectifiable with $\Hn(E)>0$, the set of all $f$ with $\Hn(f(E))>0$ is residual in $\Lipo$.

These results provide a replacement for the Besicovitch-Federer projection theorem in arbitrary metric spaces, which is known to be false outside of Euclidean spaces.
\end{abstract}

\thanks{}

\date{}
\author{David Bate}
\address{P.O. Box 68 (Gustaf H\"allstr\"omin katu 2b), FI-00014 University of Helsinki}
\email{david.bate@warwick.ac.uk}

\title[Rectifiability and perturbations]{Purely unrectifiable metric spaces and perturbations of Lipschitz functions}

\maketitle
\section{Introduction}

Recall that a subset of Euclidean space is \emph{$n$-rectifiable} if it can be covered, up to a set of $\Hn$ measure zero, by a countable number of Lipschitz (or equivalently $\mathcal C^{1}$) images of $\R^{n}$ (throughout this paper, $\Hn$ denotes the $n$-dimensional Hausdorff measure).
A set is \emph{purely $n$-unrectifiable} if all of its $n$-rectifiable subsets have $\Hn$ measure zero.

Rectifiable and purely unrectifiable sets are a central object of study in geometric measure theory, and a fundamental description of them is given by the Besicovitch-Federer projection theorem \cite{Mattila_1995}.
It states that, for a purely $n$-unrectifiable $S\subset \R^{m}$ with $\Hn(S)<\infty$, for almost every $n$-dimensional subspace $V$ of $\R^{m}$, the orthogonal projection of $S$ onto $V$ has $n$-dimensional Lebesgue measure zero.
The converse statement is an easy consequence of the Rademacher differentiation theorem: if a set is not purely $n$-unrectifiable then it contains a rectifiable subset of positive measure which has at least one $n$-dimensional approximate tangent plane.
Any projection onto a plane not orthogonal to this tangent plane has positive measure and in particular, almost every projection has positive measure.

The past several decades have seen significant activity in analysis and geometry in general metric spaces.
In particular, we mention the works of Ambrosio \cite{Ambrosio_1990}, Preiss and Tišer \cite{preiss-tiser} and Kirchheim \cite{kirchheim-regularity}, which were amongst the first to show that ideas from classical geometric measure theory generalise to an arbitrary metric space, and the later work of Ambrosio and Kirchheim \cite{ambrosio-kirchheim-currents,ambker-rectifiablesets}.
One is quickly lead to ask if a counter part to the Besicovitch-Federer projection theorem holds in this setting.
Of course, in the purely metric setting, one must interpret a \emph{projection} appropriately.
One approach is to assume additional geometric structure on the metric space that leads to an interpretation of a projection.
In this case, some positive, yet specific, results are known \cite{Brothers_1969,Brothers_1969_2,12053010,Hovila_2012}.
On the other hand, for the most general interpretation, which considers linear mappings on an infinite dimensional Banach space containing (an embedding of) the metric space, it is known that the projection theorem completely fails:
continuing from the work of De Pauw \cite{depauw-besicovitch}, Bate, Csörnyei and Wilson \cite{Bate_2017} construct, in any separable infinite dimensional Banach space $X$, a purely 1-unrectifiable set $S$ of finite $\Ho$ measure for which \emph{every} continuous linear $0 \neq T \colon X \to \R$ has $\mathcal L^1(T(S))>0$.
Thus, outside of the Euclidean setting, it is not sufficient to consider only linear mappings to Euclidean space in order to describe rectifiability.

In the metric setting, it is natural to consider \emph{Lipschitz} mappings to Euclidean space.
Indeed, this is exactly the approach taken in Cheeger's generalisation of Rademacher's theorem \cite{cheeger-diff}, and Ambrosio and Kirchheim's generalisation of currents \cite{ambrosio-kirchheim-currents}, to metric spaces.
One of the main results of this paper is to prove a suitable counterpart to the projection theorem in metric spaces for Lipschitz mappings into Euclidean space.

Namely, suppose that $X$ is a complete metric space and $S\subset X$ is purely $n$-unrectifiable with finite $\Hn$ measure and positive lower density at almost every point (see below).
The set of all bounded 1-Lipschitz functions on $X$ into some fixed Euclidean space, equipped with the supremum norm, is a complete metric space and so we can consider a residual (or comeagre) set of 1-Lipschitz functions, and such a set forms a suitable notion of a ``generic'' or ``typical'' 1-Lipschitz function.
One of the main results of this paper to show that a typical 1-Lipschitz function on $X$ maps $S$ to a set of $\Hn$ measure zero.
Conversely, it is shown that a typical 1-Lipschitz image of an $n$-rectifiable subset of $X$ has positive $\Hn$ measure.
These results are new even when $X$ is a Euclidean space.

To describe these results in more detail, recall that a subset $E$ of a metric space is $n$-rectifiable (see Definition \ref{def:rectifiable}) if it can be covered, up to a set of $\Hn$ measure zero, by a countable number of Lipschitz images of \emph{subsets} of $\R^{n}$ (considering subsets of $\R^{n}$ allows us to avoid topological obstructions).
By a result of Kirchheim \cite{kirchheim-regularity} (see Lemma \ref{lem:kirchheim}), we obtain an equivalent definition if we require \emph{biLipschitz} images of subsets of $\R^{n}$.
As for the classical case, a subset $S$ is purely $n$-unrectifiable if all of its $n$-rectifiable subsets have $\Hn$ measure zero, and any metric space $X$ with $\Hn(X)<\infty$ can be decomposed into Borel sets $E$ and $S$ where $E$ is $n$-rectifiable and $S$ is purely $n$-unrectifiable.

In \cite{kirchheim-regularity} a regularity and metric differentiation theory of rectifiable sets is given.
This was extended be Ambrosio and Kirchheim \cite{ambker-rectifiablesets} to a notion of a weak tangent plane to a rectifiable set.
Many properties of rectifiable subsets of Euclidean space can be generalised, with suitable interpretation, to the metric setting using these results.
However, positive results for purely unrectifiable subsets of metric spaces remain elusive.

We will study purely unrectifiable metric spaces by considering Lipschitz images into a Euclidean space.
Given a metric space $X$, let $\Lipo$ be the collection of all bounded 1-Lipschitz functions $f\colon X \to \R^{m}$ equipped with the supremum norm.
A subset of $\Lipo$ is \emph{residual} if it contains a countable intersection of dense open sets.
Since $\Lipo$ is complete, the Baire category theorem states that residual subsets of $\Lipo$ are dense and, since they are closed under taking countable intersections, naturally form a suitable notion of a generic Lipschitz function.

One of the main results of this paper is the following (see Theorems \ref{thm:residual-euclidean-target} and \ref{thm:positive-measure-is-open-dense}).
\begin{theorem}\label{thm:intro-main-theorem-1}
Let $X$ be a complete metric space and $S\subset X$ purely $n$-unrectifiable such that $\Hn(S)<\infty$ and
\begin{equation}
  \label{eq:lower-density}
  \liminf_{r\to 0}\frac{\Hn(B(x,r)\cap S)}{r^n}>0 \quad \text{for } \Hn\text{-a.e. } x\in S.\tag{$*$}
\end{equation}
The set of all $f\in \Lipo$ with $\Hn(f(S)) = 0$ is residual.
Conversely, if $E\subset X$ is $n$-rectifiable with $\Hn(E)>0$, the set of $f\in \Lipo$ with $\Hn(f(E))>0$ is residual.
\end{theorem}

The approach to proving this result is very general and we are able to remove the assumption \eqref{eq:lower-density} in various circumstances.
First, if $S$ is a subset of some Euclidean space, then \eqref{eq:lower-density} is not necessary (see Theorem \ref{thm:residual-euclidean-domain}).
Secondly, if $n=1$ or, more generally, $S$ is purely 1-unrectifiable, then \eqref{eq:lower-density} is not necessary (see Theorem \ref{thm:residual-domain-1pu}).
Finally, using a recently announced result of Csörnyei and Jones, it is possible to show that \eqref{eq:lower-density} is never necessary (see Remark \ref{rmk:lower-density-unnec2}).
Further, our approach applies to sets of fractional dimension.
We are able to show that for any subset $S$ of a metric space with $\Hs(S)<\infty$ for $s\not\in \N$, a typical $f\in \Lipo$ has $\Hs(f(S))=0$ (see Theorem \ref{thm:fractional-dimension}).

The conclusion of Theorem \ref{thm:intro-main-theorem-1} is related to the notion of a \emph{strongly unrectifiable} set introduced by Ambrosio and Kirchheim \cite{ambker-rectifiablesets}.
A metric space of finite $\Hn$ measure is said to be strongly $n$-unrectifiable if \emph{every} Lipschitz mapping into some Euclidean space has $\Hn$ measure zero image.
In \cite{ambker-rectifiablesets}, a construction of a strongly $n$-unrectifiable set is given for any $n\in \N$, based on an unpublished work of Konyagin.
An earlier construction of a purely 1-unrectifiable set of positive and finite $\Ho$ measure for which all real valued Lipschitz images have zero measure image was given by Vituškin, Ivanov and Melnikov \cite{MR0154965} (see also \cite{MR1313694}).
Of course, not all purely $n$-unrectifiable sets are strongly $n$-unrectifiable.
However, our main theorem shows that purely $n$-unrectifiable sets are \emph{almost} strongly $n$-unrectifiable, in a suitable sense.

The first step to prove Theorem \ref{thm:intro-main-theorem-1} (or any of the other related theorems mentioned above) is to show that any $S$ satisfying the hypotheses has a $(n-1)$-dimensional \emph{weak tangent field} with respect to \emph{any} Lipschitz $f\colon X \to \R^m$.
That is, for any Lipschitz $f\colon X \to \R^m$, after possibly removing a set of measure zero from $S$, there exists a Borel $\tau \colon S \to G(m,n-1)$ (the Grassmannian of $n-1$ dimensional subspaces of $\R^m$) such that the following holds:
for any 1-rectifiable $\gamma \subset S$, the tangent of $f(\gamma)\subset \R^m$ at a point $f(x)$, $x\in \gamma$, lies in $\tau(x)$ for $\Ho$ almost every $x\in \gamma$.
Thus, although $S$ is an $n$-dimensional set, the tangents of its 1-rectifiable subsets may only span $n-1$ dimensional subspaces.
See Definition \ref{def:atilde}.

The definition of a weak tangent field of a metric space, and its application to studying purely unrectifiable metric spaces, is new.
It is a generalisation of the weak tangent fields introduced by Alberti, Csörnyei and Preiss \cite{acp-structurenullsets,Alberti_2005,Alberti_2011} in their work on the structure of null sets in Euclidean space, where they study $(n-1)$-dimensional tangent fields of subsets of $\R^n$.
It is also related to the decomposability bundle introduced by Alberti and Marchese \cite{MR3494485}.

The construction of a weak tangent field to a purely unrectifiable subset of a metric space relies on the notion of an \emph{Alberti representation} of a metric measure space $(X,d,\mu)$, which is an integral combination 1-rectifiable measures that gives the $\mu$ measure of any Borel subset of $X$ (see Definition \ref{d:alberti-rep}).
Alberti representations were introduced in \cite{Bate_2014} to give several descriptions of those metric measure spaces that satisfy Cheeger's generalisation of Rademacher's theorem \cite{cheeger-diff}.
However, rather than their differentiability properties, we will instead be interested in the additional geometric structure that an Alberti representation provides a metric measure space.

Specifically, in Theorem \ref{thm:general-alberti-pu}, for any Lipschitz $f\colon X \to \R^m$, $S\subset X$ and $n\leq m$, we give a decomposition $S = A \cup S'$ such that $\mu\llcorner A$ has $n$ \emph{independent} Alberti representations (see Definition \ref{d:alberti-direction}) and such that $S'$ has a $(n-1)$-dimensional weak tangent field with respect to $f$.
If $S$ satisfies $\Hn(S)<\infty$ we can apply this with $\mu=\Hn$ and if in addition $S$ satisfies \eqref{eq:lower-density}, the main result of \cite{David_Bate_2017} concludes that $A$ is in fact $n$-rectifiable. 
Thus, if $S$ is purely $n$-unrectifiable, we must have $\Hn(A)=0$ and hence construct the required weak tangent field of $S$.
For subsets of Euclidean space, we will instead use the results of Alberti and Marchese \cite{MR3494485} and De Philippis and Rindler \cite{dephilippisrindler} to conclude that $\Hn(A)=0$ without assuming \eqref{eq:lower-density}.

From this point on, the proof of Theorem \ref{thm:intro-main-theorem-1} does not use the hypothesis that $S$ is purely unrectifiable and only relies upon the definition of a weak tangent field.
The main part of the argument is to construct a dense set of $\Lipo$ that maps $S$ to a set of small $\Hn$ measure.
Given $f\in \Lipo$ and $\tau$ the weak tangent field of $S$ with respect to $f$, the idea is to construct a perturbation of $f$ by locally contracting $f$ in all directions orthogonal to $\tau$.
Since $\tau$ takes values in $(n-1)$-dimensional subspaces, it is possible to reduce the $\Hn$ measure of the image of $f$ to an arbitrarily small value.
Further, since $\tau$ is a weak tangent field, this can be realised as an arbitrarily small perturbation of $f$ (see Theorem \ref{thm:general-perturbation}).
Of course, it is essential that our construction does not increase the Lipschitz constant, so that the constructed perturbation of $f$ belongs to $\Lipo$.

When considering perturbations of $\R^{m}$ valued mappings of a compact metric space $X$, it is also natural to equip the image with the supremum norm.
Indeed, for any $\eps>0$, if $x_{1},\ldots,x_{m(\eps)}$ is a maximal $\eps$-net in $X$, then in a similar fashion
to the Kuratowski embedding into $\ell_{\infty}$, the mapping $X \to \ell_\infty^{m(\eps)}$ defined by
\begin{equation*}
  x \mapsto (d(x,x_{1}),d(x,x_{2}),\ldots ,d(x,x_{m(\eps)})),
\end{equation*}
is  1-Lipschitz and perturbs relative distances in $X$ by at most $\eps$.
If $X$ has a weak tangent field, then by constructing an arbitrarily small perturbation of this map as above, we obtain a mapping that perturbs all distances in $X$ by an arbitrarily small amount that also reduces $\Hn(X)$ to an arbitrarily small amount.

If this is done naively, then the Lipschitz constant of this perturbation depends on $\eps$ (due to the comparison of the Euclidean and supremum norms in $\R^{m(\eps)}$).
If, however, we take the norm into consideration when constructing this perturbation, it is possible to construct it so that the Lipschitz constant increases by a fixed factor depending only upon $n$.
This leads to the following theorem (see Theorems \ref{thm:metric-perturbation} and \ref{thm:metric-converse}).
\begin{theorem}\label{thm:intro-main-theorem-2}
Let $X$ be a compact purely $n$-unrectifiable metric space with finite $\Hn$ measure that satisfies \eqref{eq:lower-density}.
For any $\eps>0$ there exists an $L(n)$-Lipschitz $\p_{\eps}\colon X \to \ell_\infty^{m(\eps)}$ such that $\Hn(\p_\eps(X))<\eps$ and
\begin{equation}\label{eq:intro-perturb-points}
  |d(x,y)-\|\p_{\eps}(x)-\p_{\eps}(y)\|_{\infty}|<\eps \text{ for each } x,y\in X.
\end{equation}

Conversely, if $X$ is $n$-rectifiable with $\Hn(X)>0$, then
\[\inf_{L\geq 1}\lim_{\eps\to 0} \inf \Hn(\p_{\eps}(X))>0,\]
where the second infimum is taken over all $L$-Lipschitz $\p_\eps\colon X \to \ell_\infty$ satisfying \eqref{eq:intro-perturb-points}.
\end{theorem}

Simple examples show that the converse statement is false if the Lipschitz constant is unbounded as $\eps \to 0$.
Thus, it is essential to obtain an absolute bound on the Lipschitz constant in the first half of the theorem.
As for Theorem \ref{thm:intro-main-theorem-1}, controlling the Lipschitz constant in this way requires careful consideration throughout the argument.

The assumption \eqref{eq:lower-density} can be removed under the same conditions as before, and we have a corresponding statement for fractional dimensional sets (see also Theorem \ref{thm:metric-perturbation}).

Further, if $X$ is a Banach space with an unconditional basis (see Section \ref{sec:pert-sets-uncond}), it is possible to realise $\p_{\eps}$ as a genuine perturbation of $X$.
That is, $\p_\eps\colon X\to X$ with $\|\p_{\eps}(x)-x\|<\eps$ for each $x\in X$ (see Theorem \ref{thm:banach-space-perturbation}).
This is a significant generalisation of a result of Pugh \cite{pugh16:local_besic_feder_projec_theor}, who proved the result (and its converse) for Ahlfors regular subsets of Euclidean space.
Generalising this paper was the initial motivation for the work presented here.
Note however, that Pugh's proof heavily depends on the Besicovitch-Federer projection theorem, and so our approach is entirely new.
Related is the work of Gałęski \cite{Ga_ski_2017} which finds an arbitrarily small Lipschitz perturbation with measure zero image, but at the sacrifice of any control over the Lipschitz constant.

The results that perturb a purely unrectifiable subset of a Banach space in this way immediately show the existence of a dense subset of all Lipschitz functions $f \colon X \to \R^m$ that reduce the Hausdorff measure of $X$ to an arbitrarily small amount (or to zero in the case of Gałęski).
Indeed, this follows by simply pre-composing a suitable Lipschitz extension of $f$ by such a $\p_\eps$.
However, obtaining a result for \emph{residual} subsets would require $\p_\eps$ to be 1-Lipschitz.
It is not clear how to do this in general and so we primarily consider perturbing an arbitrary Lipschitz function defined on a metric space from the outset.

We summarise our construction (see Theorem \ref{thm:general-perturbation}) of a perturbation of an arbitrary Lipschitz function $F\colon X \to \R^m$, with respect to $S\subset X$ that has a weak tangent field with respect to $F$.
For simplicity, suppose that the tangent field is constant and equal to $W \in G(m,n-1)$.

Given a linear $T\colon \R^{m} \to \R$, we first construct a perturbation $\p$ of $T\circ F$ such that, in a small neighbourhood of $S$,
\begin{equation}\label{eq:intro-project-to-subspace}
|\p(x)-\p(y)| \leq \|\pi(F(x)-F(y))\| +\eps d(x,y),
\end{equation}
for $\pi$ the orthogonal projection onto $W$ and $\eps>0$ arbitrary (see Proposition \ref{prop:basic-perturbation}).
It is easy to see that we can only do this if $S$ has a weak tangent field:
if $\gamma \subset S$ is a rectifiable curve for which $(F\circ \gamma)' \not \in V$ almost everywhere,
then $\p(\gamma)$ is a curve that is much shorter than $F(\gamma)$ (becoming shorter the further that $(F\circ \gamma)'$ lies away from $W$ on average).
Thus, $\p$ would not be an arbitrarily small perturbation of $F$, since the end points of $\gamma$ are mapped much closer together under $\p$ than $F$.
With a standard approximation argument, it is possible to reach a similar conclusion if $\gamma$ is simply 1-rectifiable, rather than a rectifiable curve.
The construction given in Section \ref{s:constructing-perturbations} shows that this condition is sufficient.
It is motivated by a similar construction in \cite{Bate_2014}, though it must be modified to fit the present needs.

We then apply the previous step to coordinate functionals of $F$.
Specifically, take a basis $B$ of $\R^{m}$ that contains $(n-1)$ vectors in $W$, and perturb the coordinate functionals of $F$ in the $m-(n-1)$ directions of $B$ not in $W$, leaving the other $n-1$ directions unchanged.
Since $W$ is $n-1$ dimensional, \eqref{eq:intro-project-to-subspace} implies that $\Hn(\p(S))$ can be made arbitrarily small.

In this construction, the Lipschitz constant of $\p$ depends on the choice of $B$.
As mentioned above, for all of our main results, we must maintain a strict control of this Lipschitz constant.
When the image of $F$ is equipped with the Euclidean norm, the natural choice of an orthonormal basis for $B$ is correct.
However, when the image of $F$ is equipped with a non-Euclidean norm, a more careful choice is required.
Therefore, before concluding with the final step of the construction, we analyse the target norm for a suitable collection of coordinate functionals (see Definition \ref{def:tilde-k}).

As mentioned above, the converse statements are false if the Lipschitz constant of the considered perturbations is not uniformly bounded.
In our proofs of the converse statements, the uniform bound allows us to modify topological arguments to the setting of rectifiable sets.
For example, a simple topological argument shows that any continuous mapping of the unit ball in Euclidean space to itself that perturbs the boundary by a small amount has positive measure image (see Lemma \ref{lem:degree}).
If this mapping is Lipschitz, then the same is true if the entire ball is replaced by an arbitrary subset with sufficiently large measure (depending only upon the Lipschitz constant of the map, see Lemma \ref{lem:converse-dense-subset-of-ball}).
Using Kirchheim's description of rectifiable sets \cite{kirchheim-regularity} (see Lemma \ref{lem:kirchheim}), this can be used to deduce the required statements about Lipschitz images of rectifiable sets.

This topological observation also leads to the following consequence of Theorem \ref{thm:intro-main-theorem-1}:
any curve (i.e. continuous image of an interval) with distinct endpoints and $\sigma$-finite $\Ho$-measure contains a rectifiable subset of positive measure.
Higher dimensional statements are also true, see Theorem \ref{thm:topological-consequences}.
Shortly after the first preprint of this article appeared, David and Le Donne \cite{1807.02664} used Theorem \ref{thm:intro-main-theorem-1} to give a stronger result than Theorem \ref{thm:topological-consequences} that only involves topological dimension.
In Euclidean space, these statements follow, in a similar fashion, from the Besicovitch-Federer projection theorem.

The structure of this paper is as follows.

In Section \ref{s:alberti-representations-rectifiability} we recall the definition of an Alberti representation of a metric measure space and some of their basic properties given in \cite{Bate_2014}.
We give a class of subsets of a metric measure space, the sets with a weak tangent field (see Definition \ref{def:atilde}), that determine when a metric measure space has many Alberti representations.
We also relate Alberti representations to rectifiability of metric spaces.
In particular we will use the main result from \cite{David_Bate_2017} that determines when a metric measure space with many Alberti representations is rectifiable.
In particular, these results show that purely unrectifiable metric spaces have a weak tangent field (see Theorem \ref{thm:atilde-summary}).

Section \ref{s:constructing-perturbations} we construct a perturbation of real valued functions.
Specifically, given $F\colon X \to \R^m$ Lipschitz and $S\subset X$ with a $d$-dimensional weak tangent field with respect to $F$, we construct a perturbation $\p$ of $T\circ F$, where $T\colon \R^{m}\to \R$ is an arbitrary linear function.
In a small neighbourhood of $S$, these perturbations satisfy \eqref{eq:intro-project-to-subspace}.
The results in this section use ideas from \cite{Bate_2014}, but they are modified to fit our requirements.

In Subsection \ref{sec:prop-finite-dimens} we gather properties of an arbitrary finite dimensional Banach space $V$ and use them to construct a collection of coordinate functionals of $V$.
These coordinate functionals are well behaved with respect to a given $d$ dimensional subspace $W$ of $V$.
Then, in Subsection \ref{sec:vector-constructions}, we apply the real valued construction of the previous section to each of these coordinate functionals to obtain a perturbation $\p$ of $F$.
The preliminary analysis of $V$ given in Subsection \ref{sec:prop-finite-dimens} results in a number $\tilde K(V,d)$ (see Definition \ref{def:tilde-k}).
Our construction is such that $\Lip \p$ is at most $\tilde K(V,d) \Lip F$.

We will see that $\tilde K(\R^m,d)=1$ for any $m,d\in\N$ and so, given a function in $\Lipo$, our construction produces a function in $\Lipo$.
This allows us to show that certain subsets of $\Lipo$ are dense and hence form residual sets.
This is done in Section \ref{sec:typical}.

This concludes one direction of the proof of our main theorems.
In Section \ref{sec:proof-main-theorems} we combine the results of the previous sections and state and prove these theorems.
Our constructions regarding coordinate functionals of finite dimensional Banach spaces are related to concepts from infinite dimensional geometric measure theory.
In Section \ref{sec:pert-sets-uncond} we highlight these relationships and use them to deduce a perturbation theorem for purely unrectifiable subsets of Banach spaces with an unconditional basis.

Finally, we prove various results regarding rectifiable subsets of a metric space in Section \ref{sec:converse}.

\subsection{Notation}
Throughout this paper, $(X,d)$ will denote a complete metric space.
Since any Lipschitz function may be uniquely extended to the completion of its domain, this is a natural assumption in our setting and simply alleviates issues arising from measurability.
For example, it implies that, for any $\Hs$ measurable $S\subset X$ with $\Hs(S)<\infty$, $\Hs \llcorner S$ is a finite Borel regular measure on the closure of $S$, a complete and separable metric space.
In particular, this implies that $\Hs \llcorner S$ is inner regular by compact sets.

Here and throughout, $\Hs$ will denote the $s$-dimensional Hausdorff measure on $X$ defined,
for $S\subset X$ and $s,\delta>0$, by
\begin{equation*}
  \Hs_\delta(S) = \inf\left\{\sum_{i\in\N} \diam(S_i)^s : S \subset \bigcup_{i\in\N} S_i,\ \diam(S_i)\leq \delta\right\}
\end{equation*}
and $\Hs(S) = \lim_{\delta\to 0} \Hs_\delta(S)$.

For $x\in X$ and $r>0$, $B(x,r)$ will denote the open ball of radius $r$ centred on $x$.
If $S\subset X$, $B(S,r)$ will denote the open $r$-neighbourhood of $S$ and $\overline S$ the closure of $S$.
We will write $d(x,S)$ for the infimal distance between $x$ and points of $S$.

For $(Y,\rho)$ a metric space and $L\geq 0$, a function $f\colon X \to Y$ is said to be \emph{$L$-Lipschitz} (or simply \emph{Lipschitz} if such an $L$ exists) if
\[\rho(f(x),f(y)) \leq L d(x,y)\]
for each $x,y\in X$.
We let $\Lip f$ be the least $L\geq 0$ for which $f$ is $L$-Lipschitz.
Further, if $f$ is Lipschitz, we let
\[\Lip(f,x)=\limsup_{y\to x}\frac{\rho((f(x),f(y)))}{d(x,y)},\]
the \emph{pointwise Lipschitz constant} of $f$.
We will write $\Lip(f,\cdot)$ for the function $x\mapsto \Lip(f,x)$.

We will require results from the theory of \emph{metric measure spaces}: complete metric spaces $(X,d)$ with a $\sigma$-finite Borel regular Radon measure $\mu$.
However, our only application will be to the metric measure spaces of the form $(X,d,\Hs\llcorner S)$, for $S\subset X$ $\Hs$ measurable.

We define a rectifiable set as follows.
\begin{definition}\label{def:rectifiable}
For $n\in\N$, a $\Hn$ measurable $E\subset X$ is \emph{$n$-rectifiable} if there exists a countable number of Lipschitz $f_i\colon A_i \subset \mathbb R^n \to X$ such that
\begin{equation*}
\Hn\left(E\setminus \bigcup f_i(A_i)\right)=0.
\end{equation*}
A $\Hn$ measurable $S\subset X$ is \emph{purely $n$-unrectifiable} if $\Hn(S\cap E)=0$ for every $n$-rectifiable $E\subset X$.
\end{definition}
Since $X$ is complete, an equivalent definition of rectifiable sets is obtained if we require the $A_i$ to be compact.
If $X$ is a Banach space, then by obtaining a Lipschitz extension of each $f_i$ (see \cite{Johnson_1986}), an equivalent definition is obtained by requiring each $A_i=\R^n$.

We write $G(d,m)$ for the Grassmannian of $d$-dimensional subspaces of $\R^m$.
We may sometimes write $W\leq V$ to denote that $W$ is a subspace of $V$.

Throughout this paper, the notation $\|.\|$ will refer to the intrinsic norm of a Banach space, be it the Euclidean norm on $\R^m$, the supremum norm on a set of bounded functions, the operator norm on a set of bounded linear functions or the norm of some other arbitrary Banach space.
Whenever this notation is used, the precise norm in question should be clear from the context.

\subsection{Acknowledgements}
This work was supported by the Academy of Finland projects 308510 and 307333, and the University of Helsinki project 7516125.

I would like to thank Pertti Mattila for bringing the work of Pugh to my attention and for comments on the first preprint of this article.
I would also like to thank Bruce Kleiner and David Preiss for comments that lead me to an enlightenment on the final presentation of this work.
I am grateful to Tuomas Orponen for useful discussions during the preparation of this article.
Finally, I would like to thank the referee for suggestions on how to improve the readability of the paper.

\section{Alberti representations, rectifiability and weak tangent fields} \label{s:alberti-representations-rectifiability}
We now recall the definition of an Alberti representation of a metric measure space introduced in \cite{Bate_2014}, and give conditions that ensure the existence of many \emph{independent} Alberti representations.
Following this, we give various conditions under which a metric measure space with many independent Alberti representations is in fact rectifiable.
By combining these, we develop the ideas into the notion of a \emph{weak tangent field} of a purely unrectifiable subset of a metric measure space.

\subsection{Alberti representations of a measure}
An Alberti representation of a measure is an integral representation by rectifiable curves.
One important point is that we allow these curves to be Lipschitz images of \emph{disconnected} subsets of $\R$.
This allows us to consider all metric spaces, regardless of obvious topological obstructions.
\begin{definition}\label{d:alberti-rep}
  Let $(X,d)$ be a metric space.  We define the set of \emph{curve fragments} of $X$ to be the set
  \[\Gamma(X) := \{\gamma \colon \dom\gamma \subset \R \to X : \dom\gamma \text{ compact, } \gamma \text{ biLipschitz}\}.\]
  We equip $\Gamma(X)$ with the Hausdorff metric induced by the inclusion
  \[\gamma\in \Gamma(X) \mapsto \graph \gamma \subset \R \times X.\]

  An \emph{Alberti representation} of a metric measure space $(X,d,\mu)$ consists of a probability measure $\P$ on $\Gamma(X)$ and, for each $\gamma\in \Gamma(X)$, a measure $\mu_{\gamma}\ll \Ho\llcorner\gamma$ such that
  \[\mu(B) = \int_\Gamma \mu_\gamma(B) d\P(\gamma)\]
  for each Borel $B\subset X$.
  Integrability of the integrand is assumed as a part of the definition.
\end{definition}

Alberti representations first appeared in the generality of metric spaces in \cite{Bate_2014}, where they were used to give several characterisations of Cheeger's generalisation of Rademacher's theorem.
The relationship between Alberti representations and differentiability can be seen in the following observation.

Suppose that $\gamma\in \Gamma(X)$ and $F\colon X \to \R^m$ is Lipschitz.
Then $F\circ \gamma \colon \dom\gamma \to \R^m$ and so it is differentiable at almost every point of $\dom\gamma$.
Therefore, if $\mu$ has an Alberti representation, for $\mu$ almost every $x$, there exists a curve fragment $\gamma \ni x$ for which $(F\circ\gamma)'(\gamma^{-1}(x))$ exists.
That is, $F$ has a \emph{partial derivative} at $x$.

Alternatively, although a curve fragment may not have a tangents in $X$, there exist many tangents after mapping the fragment to a Euclidean space.
This allows us to distinguish ``different'' Alberti representations: Alberti representations will be considered different if we can find a single Lipschitz map to Euclidean space that distinguishes their tangents.

\begin{definition}
  \label{def:cones}
  For $w\in \ell_2^m$ and $0<\theta<1$ define the \emph{cone centred on $w$ of width $\theta$} to be
  \[C(w,\theta):= \{v\in \R^m: v\cdot w \geq (1-\theta)\|v\|\}.\]
  We say that cones $C_1,\ldots, C_n \subset \ell_2^m$ are \emph{independent} if, for any choice of $w_i \in C_i \setminus \{0\}$ for each $1\leq i \leq n$, the $w_i$ are linearly independent.
  
  Now let $V$ be a finite dimensional Banach space.
  For $W\leq V$ a subspace, we define the ``conical complement'' of $W$ to be
  \[E(W,\theta):= \{v\in V : d(v,W) \geq (1-\theta)\|v\|\}.\]
  
  Note that both of the above sets become wider as $\theta\to 1$.
  Whilst sets of either form may be considered ``cones'', we will reserve this name, and the notation ``$C$'', for sets of the first type.
\end{definition}

\begin{definition}\label{d:alberti-direction}
  Let $(X,d)$ be a metric space, $V$ a finite dimensional Banach space, $F\colon X \to V$ Lipschitz and $D$ a set of the form $C(w,\theta)$ (if $V=\ell_2^m$) or $E(W,\theta)$.
  We say that a curve fragment $\gamma \in \Gamma$ is in the \emph{$F$-direction} of $D$ if
  \[\Fogt \in D\setminus \{0\}\]
  for $\Ho$-a.e.\ $t\in \dom\gamma$.
  Further, an Alberti representation $(\P,\{\mu_\gamma\})$ of $(X,d,\mu)$ is in the $F$-direction of a cone $C$ if $\P$-a.e.\ $\gamma\in \Gamma$ is in the $F$-direction of $C$.

  Finally, Alberti representations $\mathcal A_1, \ldots, \mathcal A_n$ of $(X,d,\mu)$ are \emph{independent} if there exist an $m\in \N$, a Lipschitz $F\colon X\to \mathbb \ell_2^m$ and independent cones $C_1, \ldots, C_n \subset \ell_2^m$ such that $\mathcal A_i$ is in the $F$-direction of $C_i$ for each $1\leq i \leq n$.
  In this case, we say that the Alberti representations are $F$-independent.
\end{definition}

In the definition of independent Alberti representations, we could permit a Lipschitz function taking values in any finite dimensional Banach space. However, its rather straightforward to see that this definition is equivalent to the one given above (up to a countable decomposition of the support of the measure).
Thus, for compatibility with \cite{Bate_2014}, we will only consider a $\ell_2^m$ valued function.
We do, however, require the definition of $E(W,\theta)$ for arbitrary finite dimensional Banach spaces.
For the remainder of this section, we will write $\R^m$ for $\ell^m_2$.

This definition of independent Alberti representations differs slightly from the definition given in \cite{Bate_2014}.
There, the definition requires the dimension of the image ($m$) and the number of Alberti representations ($n$) to agree.
However, it is easy to see that these definitions are equivalent.
Indeed, $F\colon X \to \R^m$ is Lipschitz and Alberti representations $\mathcal A_1, \ldots, \mathcal A_n$ are in the $F$-direction of $C(w_1,\theta),\ldots, C(w_n,\theta)$, let $\pi$ be the orthogonal projection onto the span of the $w_i$.
Then it is easy to check that the $\mathcal A_i$ are in the $\pi\circ F$ direction of the $\pi(C_i)$ and that the $\pi(C_i)$ are independent cones.

Although it is a small change to the definition, considering a smaller number of Alberti representations than the dimension of the image is required for us to develop the notion of a weak tangent field of a metric space.

One of the main results of \cite{Bate_2014} gives an equivalence between Cheeger's generalisation of Rademacher's theorem and the existence of many independent Alberti representations of a metric measure space.
Further, independent to interests in differentiability, an Alberti representation is a new concept to provide additional structure to a metric measure space.
In subsection \ref{ss:alberti-rect} below, we will give various results that show when a metric measure space $(X,d,\Hn)$ with $n$-independent Alberti representations is, in fact, $n$-rectifiable.
For the rest of this subsection, we will develop conditions that ensure that a metric measure space has many independent Alberti representations, so that these results can be applied.

First suppose that $w\in \R^n$, $F\colon X \to \R^m$ is Lipschitz and $\mu$ has an Alberti representation in the $F$-direction of $C(w,\theta)$.
Then necessarily, any Borel $S\subset X$ with $\Ho(\gamma\cap S)=0$ for each $\gamma \in \Gamma$ in the $F$-direction of $C(w,\theta)$, must have $\mu(S)=0$.
This condition is also sufficient for the existence of an Alberti representation.
\begin{lemma}[\cite{Bate_2014}, Corollary 5.8]\label{lem:one-rep}
  Let $(X,d,\mu)$ be a metric measure space, $F\colon X \to \R^m$ Lipschitz and $C\subset \R^m$ a cone.
  There exists a Borel decomposition $X=A \cup S$ such that $\mu\llcorner A$ has an Alberti representation in the $F$-direction of $C$ and $S$ satisfies $\Ho(\gamma\cap S)=0$ for each $\gamma\in \Gamma$ in the $F$-direction of $C$.
\end{lemma}
We also require the following result, which allows us to \emph{refine} the directions of an Alberti representation.
\begin{lemma}[\cite{Bate_2014}, Corollary 5.9]\label{lem:refine}
  Let $(X,d,\mu)$ be a metric measure space, $F\colon X \to \R^m$ Lipschitz and $C\subset \R^m$ a cone.
  Suppose that, for some cone $C\subset \R^m$, $\mu\llcorner A$ has an Alberti representation in the $F$-direction of $C$.
  Then, for any countable collection of cones $C_k$ with
  \[\bigcup_{k\in\N}\operatorname{interior}(C_k) \supset C\setminus \{0\},\]
  there exists a countable Borel decomposition $A=\cup_k A_k$ such that each $\mu\llcorner A_k$ has an Alberti representation in the $F$-direction of $C_k$.
\end{lemma}
We will use this lemma in the following way.
Suppose that $\mu\llcorner A$ has Alberti representations in the $F$-direction of independent cones $C_1,\ldots,C_d$.
For any $0<\eps<1$, we may cover each $C_i$ by the interior of a finite number of cones $C_i^j$ of width $\eps$ such that any choice $C_1^{j_1},\ldots,C_d^{j_d}$ is also independent.
By applying the lemma to these collections, we see that there exists a finite Borel decomposition $A=\cup_i A_i$ such that each $\mu\llcorner A_i$ has $d$ $F$-independent Alberti representations in the $F$-direction of cones of width $\eps$.

It is possible to define a collection $\tilde A(F)$ of subsets of $X$ that extends the decomposition given in Lemma \ref{lem:one-rep} in the following way:
there exists a decomposition $X=S\cup \cup_i U_i$ such that $S\in \tilde A(F)$ and each $\mu\llcorner U_i$ has $m$ $F$-independent Alberti representations (see \cite[Definition 5.11, Proposition 5.13]{Bate_2014}).
However, as mentioned above, it will be necessary for us consider the case when $\mu$ has $d$ $F$-independent Alberti representations, for $d\leq m$.
Our first task is to give a suitable decomposition in this case.

We begin with the following.
\begin{lemma}\label{lem:depomposition-subspace-complement}
  Let $(X,d,\mu)$ be a metric measure space, $F\colon X \to \R^m$ Lipschitz and, for some $0 \leq d \leq m$, let $W\leq \R^m$ be a $d$-dimensional subspace.
  For any Borel $U\subset X$, $0<\theta<1$ and $0<\eps< 1-\theta$, there exists a Borel decomposition
  \[U = S \cup U_1 \cup \ldots \cup U_N\]
  such that $\Ho(\gamma \cap S)=0$ for each $\gamma \in \Gamma$ in the $F$-direction of $E(W,\theta)$ and each $\mu\llcorner U_i$ has an Alberti representation in the $F$-direction of some cone $C_i \subset E(W,\theta+\eps)$.
\end{lemma}

\begin{proof}
  Cover $E(W,\theta)$ by cones $C_1, \ldots, C_N \subset E(W,\theta+\eps)$ and for each $1\leq i \leq N$ apply Lemma \ref{lem:one-rep} to obtain a decomposition $U = U_i \cup S_i$ where $\mu \llcorner U_i$ has an Alberti representation in the $F$-direction of $C_i$ and $\Ho(\gamma\cap S_i)=0$ for each $\gamma$ in the $F$-direction of $C_i$.

  Observe that $S:= \cap_i S_i$ satisfies $\Ho(\gamma\cap S)=0$ for any $\gamma$ in the $F$-direction of $E(W,\theta)$.
  Indeed, if $\gamma$ is in the $F$-direction of $E(W,\theta)$, there exists a decomposition $\gamma = \gamma_1\cup \ldots \cup \gamma_N$ so that each $\gamma_i$ is in the $F$-direction of $C_i$.
  Thus $\Ho(\gamma_i \cap S)=0$ for each $1\leq i \leq N$ and so $\Ho(\gamma \cap S)=0$.
  Therefore, $U=S\cup U_1\cup \ldots \cup U_N$ is the required decomposition.
\end{proof}

Next we define the sets that generalise the $\tilde A(F)$ sets mentioned above.
We will see that these are precisely those sets with a \emph{weak tangent field}.
Weak tangent fields were first defined in the works of Alberti, Csörnyei and Preiss \cite{Alberti_2011, Alberti_2005, acp-structurenullsets} where many aspects of the classical theory of Alberti representations appears.
In these papers it is shown that any Lebesgue null set in the plane has a weak tangent field.
Furthermore, the relationship between weak tangent fields and various questions in geometric measure theory is established.
\begin{definition}
  \label{def:atilde}
  Fix a finite dimensional Banach space $V$, a Lipschitz $F\colon X \to V$ and an integer $d \leq \dim V$.

  For $0<\theta<1$ we define $\tilde A(F, d, \theta)$ to be the set of all $S\subset X$ for which there exists a Borel decomposition $S=S_{1} \cup \ldots \cup S_{M}$ and $d$-dimensional subspaces $W_{i} \leq V$ such that, for each $1\leq i \leq M$, $\Ho(\gamma\cap S_{i})= 0$ for every $\gamma\in \Gamma$ in the $F$-direction of $E(W_{i},\theta)$.
  Further, we define $\tilde A(F,d)$ to be the set of all $S\subset X$ that belong to $\tilde A(F,d,\theta)$ for each $0<\theta<1$.

  For $m \in \N$, let $\mathcal C$ be the collection of closed, conical subsets of $\R^m$ (that is, closed sets that are closed under multiplication by scalars).
  We define a metric on $\mathcal C$ by setting $d(V,W)$ to be the Hausdorff distance between $V\cap \mathbb S^{m-1}$ and $W\cap \mathbb S^{m-1}$.
  Note that for any integer $d\leq m$, $G(m,d)$ is a closed subset of $\mathcal C$.

  Let $S\subset X$ be Borel.
  A Borel $\tau \colon S \to G(m,d)$ is a \emph{$d$-dimensional weak tangent field to $S$ with respect to $F$} if, for every $\gamma\in\Gamma(X)$,
  \begin{equation*}
    (F\circ\gamma)'(t) \in \tau(\gamma(t)) \quad \text{for } \Ho \text{-a.e. } t\in\gamma^{-1}(S).
  \end{equation*}
\end{definition}
Note that the sets $\tilde A(F,d,\theta)$ decrease as $\theta$ increases to 1, and that any Borel subset of a $\tilde A(F,d,\theta)$ set is also in $\tilde A(F,d,\theta)$.
Also, $\tilde A(F,d,\theta) \subset \tilde A(F,d',\theta)$ if $d\leq d'$.
Further, by the compactness of $\mathbb S^{m-1}$, an equivalent definition is obtained if we allow \emph{countable} decompositions of an $\tilde A(F,d,\theta)$ set, rather than finite decompositions.
Thus, $\tilde A(F,d,\theta)$ and hence $\tilde A(F,\theta)$ sets are closed under countable unions.

The $\tilde A(\phi)$ sets of \cite{Bate_2014} are essentially $\tilde A(\phi,n-1)$ sets and
the weak tangent field introduced by Alberti, Csörnyei and Preiss for a set $S\subset \R^{n}$ is what we call an $(n-1)$-dimensional weak tangent field with respect to the identity.

It is easy to see the connection between weak tangent fields and $\tilde A$ sets.
The only technical point is to construct a tangent field in a Borel regular way.
First the simple direction.
\begin{lemma}\label{lem:tangent-field-gives-atilde}
	For $F\colon X \to \R^m$ Lipschitz, let $S\subset X$ have a $d$-dimensional weak tangent field with respect to $F$.
	Then $S\in \tilde A(F,d)$.
\end{lemma}

\begin{proof}
  Suppose that $\tau \colon S \to G(m,d)$ is a $d$-dimensional weak tangent field with respect to $F$ and let $0<\theta<1$.
  Let $W_1,\ldots,W_M \in G(m,d)$ such that
  \[\bigcap_{i=1}^M E(W_i,\theta)=\{0\}\]
  and, for each $1\leq i \leq M$, let $S_{i}$ be those $x\in S$ for which $\tau(x) \subset \R^m \setminus E(W_{i},\theta)$, a Borel set.
  Then, if $\gamma \in \Gamma(X)$, $(F\circ \gamma)'(t) \in \R^m \setminus E(W_{i},\theta)$ for almost every $t \in \gamma^{-1}(S_i)$.
  Therefore, $\Ho(\gamma\cap S_i)=0$ for each $\gamma\in \Gamma$ in the $F$-direction of $E(W_{i},\theta)$.
\end{proof}

For the other direction, we must take a little care to construct the weak tangent field in a Borel way.
\begin{lemma}
  \label{lemma:atlide-gives-tangent-field}
  For $F\colon X \to \R^m$ Lipschitz, let $S \in \tilde A(F,d)$.
  Then $S$ has a $d$-dimensional weak tangent field with respect to $F$.
\end{lemma}

\begin{proof}
  For each $j\in \N$ let
    \begin{equation*}
    S = \bigcup_{i=1}^{M_{j}}S_{i,j}
  \end{equation*}
  be a disjoint Borel decomposition given by the definition of an $\tilde A$ set with the choice $\theta=1/j$, where $W_{i,j}\in G(m,d)$.
  To define a weak tangent field with respect to $F$,
  for each $x\in S$ and $j\in\N$, let $i(x,j)\in \N$ such that $x \in S_{i(x,j),j}$.
  For each $n\in \N$ define
  \begin{equation*}
    L_n(x) = \bigcap_{j=1}^n C(W_{i(x,j),j},1/j) \in \mathcal C
  \end{equation*}
  and
  \begin{equation*}
    L(x) = \bigcap_{j\in \N}C(W_{i(x,j),j},1/j) \in \mathcal C,
  \end{equation*}
  for $C(W,\theta)$ the closure of $\R^m\setminus E(W,\theta)$ (it is a ``cone'' around $W$).

  First observe that, for any $\gamma\in \Gamma(X)$, $(F\circ\gamma)'(t) \in L(\gamma(t))$ for almost every $t\in \gamma^{-1}(S)$.
  Indeed, for each $j\in \N$ and $1\leq i \leq M_{j}$, for almost every $t\in\gamma^{-1}(S_{i,j})$, $(F\circ\gamma)'(t) \in C(W_{i,j},1/j)$.
  That is, for almost every $t \in \gamma^{-1}(S)$ and every $j\in \N$, $(F\circ\gamma)'(t)\in C(W_{i(\gamma(t),j),j},1/j)$.
  Therefore, for a full measure subset of $\gamma^{-1}(S)$, $(F\circ\gamma)'(t)\in L(\gamma(t))$.
  Of course, $L(x)$ may not belong to $G(m,d)$, and so we must find a weak tangent field $\tau$ that contains $L$ at almost every point.

  However, $L(x)$ is contained in a $d$-dimensional subspace for each $x\in S$.
  Indeed, let $W_{i(x,j_k),j_k}\to W$ be any convergent subsequence as $k\to \infty$.
  Then
  \[L(x) \in L_k(x) \subset C(W_{i(x,j_k),j_k},1/j_k) \subset C(W,d_k),\]
  where $d_k \to 0$ as $j_k \to 0$, so that $L(x)\subset W$.

  Moreover, $L(x)$ is a Borel function, since, for each $x\in S$, $L_n(x)\to L(x)$ as $n\to \infty$.
  Indeed, since $L(x) \subset L_n(x)$ for each $n\in \N$, if $L_n(x)\not\to L(x)$, there exist some $\eps>0$ and a sequence $y_n \in \mathbb S^{n-1}\cap L_n(x)$ with
  \[y_n \not \in B(\mathbb S^{m-1}\cap L(x),\eps)\]
  for each $n\in \N$.
  By the compactness of $\mathbb S^{m-1}$, we may suppose that
  \[y_n\to y \not \in B(\mathbb S^{m-1}\cap L(x),\eps)\]
  as $n\to \infty$.
  Since $L_n(x)$ decreases as $n$ increases, $y_n \in L_{n'}(x)$ whenever $n'\leq n$, and so $y\in L_{n'}(x)$ for each $n'\in\N$.
  Therefore, $y\in L(x)$, a contradiction.

  Thus, we can set $\tau(x)$ to be the span of $L(x)$ and, wherever necessary, extend it to a $d$-dimensional subspace in a Borel way.
\end{proof}

Next we generalise \cite[Proposition 5.13]{Bate_2014}.
Although it is possible to deduce this result from \cite[Proposition 5.13]{Bate_2014}, because of several technical details in the statement of that proposition, it is simpler to give a direct proof.
\begin{proposition}\label{prop:atilde-decomposition}
  Let $(X,d,\mu)$ be a metric measure space, $F\colon X \to \R^m$ Lipschitz and $0\leq d <m$ an integer.
  There exists a Borel decomposition
  \[X=S \cup \bigcup_{j\in \N}U_j\]
  where $S\in \tilde A(F,d)$ and each $\mu\llcorner U_j$ has $d+1$ $F$-independent Alberti representations.
\end{proposition}

\begin{proof}
  Fix $0<\theta<1$ and choose an arbitrary $d$-dimensional subspace $W\leq \R^m$ and apply Lemma \ref{lem:depomposition-subspace-complement} to obtain a Borel decomposition $U= S \cup U_1 \cup \ldots \cup U_N$ where $\Ho(\gamma\cap S)=0$ for each $\gamma\in\Gamma$ in the $F$-direction of $E(W,\theta)$ and each $\mu\llcorner U_j$ has an Alberti representation in the $F$-direction of some cone $C_j\subset \R^m$.
  In particular, $S\in \tilde A(F,d,\theta)$.

  If $d=0$ then we are done.
  Otherwise, suppose that, for some $0< i \leq d$, there exists a Borel decomposition $U= S \cup U_1 \cup\ldots \cup U_N$ such that each $\mu\llcorner U_j$ has $i$ $F$-independent Alberti representations and $S\in \tilde A(F,d,\theta)$.
  By applying Lemma \ref{lem:refine} and taking a further decomposition if necessary, we may suppose that each Alberti representation of the $\mu\llcorner U_j$ are in the $F$-direction of cones of width $0<\alpha< \sqrt{1-\theta^2}/2$.

  For a moment fix $1\leq j \leq N$ and let $C(w_1,\alpha),\ldots,C(w_i,\alpha)$ be independent cones that define the $F$-direction of the Alberti representations of $\mu\llcorner U_j$.
  By applying Lemma \ref{lem:depomposition-subspace-complement} to a $d$-dimensional subspace $W$ containing $w_1,\ldots, w_i$, we obtain a decomposition $U_j = S_j \cup U_1^j \cup \ldots \cup U_{M_j}^j$ where $S_j \in \tilde A(F,d,\theta)$ and each $\mu\llcorner U_k^j$ has an Alberti representation in the $F$-direction of some cone $C \subset E(W,\theta+\eps)$ in addition to the other $i$ Alberti representations.
  Since $\alpha< \sqrt{1-\theta^2}/2$, $\eps>0$ may be chosen so that $C, C_1,\ldots, C_i$ forms an independent collection of cones, and hence so that each $\mu\llcorner U_k^j$ has $i+1$ $F$-independent Alberti representations.

  Since $S':=S\cup S_1 \cup \ldots \cup S_N \in \tilde A(F,d,\theta)$, this gives a Borel decomposition $U=S'\cup_{j,k} U_{k}^j$ where $S' \in \tilde A(F,d,\theta)$ and each $\mu \llcorner U_k^j$ has $i+1$ $F$-independent Alberti representations.

  Repeating this process $d-1$ times gives a decomposition $X \setminus S^\theta =\cup_{j}U_j^\theta$ where each $\mu\llcorner U_j$ has $d+1$ Alberti representations and $S^\theta \in \tilde A(F,d,\theta)$.
  Repeating this for $\theta_i \to 1$ and setting $S=\cap_{i}S^{\theta_i} \in \tilde A(F,d)$ gives a decomposition
  \[X=S\cup \bigcup_{i,j}U_j^{\theta_i}\]
  of the required form.
\end{proof}

We also obtain the following generalisation of \cite[Theorem 5.14]{Bate_2014}.
\begin{theorem}
  \label{thm:general-alberti-pu}
  Let $(X,d,\mu)$ be a metric measure space, $F\colon X \to \R^{m}$ Lipschitz and $d<m$ and integer.
  \begin{enumerate}
  \item \label{item:atilde-upperbound} For every positive measure Borel subset $X'$ of $X$, $\mu\llcorner X'$ has at most $d$ $F$-independent Alberti representations if and only if there exists $N\subset X$ with $\mu(N)=0$ and $X\setminus N \in \tilde A(F,d)$.
  \item \label{item:atlide-exact} There exists a decomposition $X=\cup_i X_i$ so that each $\mu\llcorner X_i$ has $d+1$ $F$-independent Alberti representations if and only if each $\tilde A(F,d)$ subset of $X$ is $\mu$-null.
  \end{enumerate}
\end{theorem}

\begin{proof}
  We first prove \eqref{item:atilde-upperbound}.
  One direction follows from the previous proposition.
  Indeed, if $U_j$ are as in the conclusion of the proposition then, by assumption, each $U_j$ must have $\mu$ measure zero.
  Therefore, setting $N= \cup_i N_i$, a $\mu$-null set completes this direction.

  We prove the other direction by contradiction.
  Suppose that $X'\subset X$ has positive measure and $d+1$ $F$-independent Alberti representations in the direction of cones $C_1,\ldots, C_{d+1}\subset \R^m$.
  Choose $0<\theta<1$ sufficiently large (depending only upon the configuration of the $C_i$ and $m$) such that, for any $d$-dimensional subspace $W\leq \R^m$, $E(W,\theta)$ contains at least one of the $C_i$.

  Since there exists a $\mu$-null set $N$ such that $X\setminus N \in \tilde A(F,d)$, there exists a positive measure subset $Y$ of $X'$ and a $d$-dimensional subspace $W\leq \R^m$ such that $\Ho(\gamma\cap Y')=0$ for each $\gamma \in \Gamma$ in the $F$-direction of $E(W,\theta)$.
  By the choice of $\theta$ above, there exists some $C_i \subset E(W,\theta)$ and so, since $\mu\llcorner X'$ has an Alberti representation in the $F$-direction of $C_i$, we see that $\mu(Y)=0$, a contradiction.

  One direction of \eqref{item:atlide-exact} also follows from the previous proposition.
  For the other direction, suppose that $X=\cup_i X_i$ is such a decomposition and let $S\in \tilde A(F,d)$.
  By applying \eqref{item:atilde-upperbound} to the metric measure space $(X,d,\mu\llcorner S)$, we see that every positive measure subset of $S$ can have at most $d$ $F$-independent Alberti representations.
  However, if $\mu(S)>0$, there exists some $i\in\N$ with $\mu(S\cap X_i)>0$ and hence $S\cap X_i$ is a positive measure subset of $S$ with $d+1$ $F$-independent Alberti representations, a contradiction.
\end{proof}

\subsection{Alberti representations and rectifiability}
\label{ss:alberti-rect}
In this subsection we will give conditions that ensure that a metric measure space with $n$ independent Alberti representations is $n$-rectifiable.
By combining these conditions with the results from the previous subsection, we will obtain a relationship between purely unrectifiable sets and $\tilde A$ sets.

The main result we will use is the following.
\begin{theorem}[{\cite[Theorem 1.2]{David_Bate_2017}}]\label{t:alberti-rect}
  Suppose that a metric measure space $(X,d,\mu)$ satisfies
  \[0 <\liminf_{r\to 0}\frac{\mu(B(x,r))}{r^n} \leq \limsup_{r\to 0}\frac{\mu(B(x,r))}{r^n} < \infty  \quad \text{for } \mu\text{-a.e.\ } x\in X\]
  and has $n$ independent Alberti representations.
  Then there exists a Borel $N\subset X$ with $\mu(N)=0$ such that $X\setminus N$ is $n$-rectifiable.
\end{theorem}

We can easily transform the previous result into one about purely $n$-unrectifiable sets.
\begin{corollary}\label{c:alberti-pu}
  Let $S\subset X$ have $\Hn(S)<\infty$, be purely $n$-unrectifiable and satisfy
  \begin{equation*}\tag{$*$}
  	 \liminf_{r\to 0}\frac{\Hn(B(x,r)\cap S)}{r^n}>0 \quad \text{for } \Hn\text{-a.e. } x\in S.
  \end{equation*}
  Then for every Borel $S'\subset S$ of positive $\Hn$ measure, $\Hn \llcorner S'$ has at most $n-1$ independent Alberti representations.
\end{corollary}

\begin{proof}
  Let $S'\subset S$ be Borel.
  Since $S$ has finite $\Hn$ measure, \cite[Theorem 2.10.18]{federer} implies
  \begin{equation*}
    \limsup_{r\to 0}\frac{\Hn(B(x,r)\cap S)}{(2r)^n}\leq 1 \quad \text{for } \Hn\text{-a.e.\ } x\in S
  \end{equation*}
  and
  \begin{equation*}
    \limsup_{r\to 0}\frac{\Hn(B(x,r)\cap (S\setminus S'))}{r^n}=0 \quad \text{for } \Hn\text{-a.e. } x\in S'.
  \end{equation*}
  In particular, by combining with \eqref{eq:lower-density},
  \begin{equation*}
    \liminf_{r\to 0}\frac{\Hn(B(x,r)\cap S')}{r^n} >0 \quad \text{for } \Hn\text{-a.e. } x\in S'.
  \end{equation*}
  Therefore, if $\Hn\llcorner S'$ has $n$ independent Alberti representations, $(X,d,\Hn\llcorner S')$ satisfies the hypotheses of Theorem \ref{t:alberti-rect} and so $S'$ is $n$-rectifiable.
  In particular, since $S$ is purely $n$-unrectifiable, we must have $\Hn(S')=0$.
\end{proof}

There are many situations when the lower density assumption \eqref{eq:lower-density} is not necessary.
First, we mention that it is never necessary.
We will not prove this, but mention it to set the scope for the results of this paper.
\begin{remark}
  \label{rmk:lower-density-unnec}
  Using very deep results regarding the structure of null sets in $\R^{n}$ recently announced by Csörnyei and Jones \cite{csornyei-jones}, it is possible to show that any $(X,d,\Hn)$ with $n$-independent Alberti representations necessarily satisfies \eqref{eq:lower-density}.
  In particular, $X$ is $n$-rectifiable, and Corollary \ref{c:alberti-pu} is true without the assumption \eqref{eq:lower-density}.
  If $n=2$, this can be deduced from the work of Alberti, Csörnyei and Preiss \cite{Alberti_2005}.
  This will appear in future work of myself and T.\ Orponen.
\end{remark}

Without the announcement of Csörnyei and Jones, it is still possible to remove the assumption \eqref{eq:lower-density} in many situations.

First, observe that it is not necessary for 1-dimensional sets.
\begin{observation}
  \label{obs:pu-no-density}
  For any purely 1-unrectifiable metric space $X$, a (non-trivial) measure $\mu$ on $X$ cannot have \emph{any} Alberti representations, and in fact $X\in \tilde A(F,0)$ for any Lipschitz $F\colon X \to \R^m$ and any $m\in \N$.
\end{observation}

Using the theory of Alberti representations in Euclidean space given by De Philippis and Rindler \cite{dephilippisrindler} and Alberti and Marchese \cite{MR3494485}, we can remove the assumption \eqref{eq:lower-density} when metric space is a subset of some Euclidean space.
Specifically, we will use the following theorem.
\begin{theorem}[Lemmas 3.2, 3.3 \cite{conjectureofcheeger}]
  \label{thm:dephilippis-marchese-rindler}
  Let $(X,d,\mu)$ be a metric measure space and $\phi\colon X \to \R^{n}$ Lipschitz such that $\mu$ has $n$ $\phi$-independent Alberti representations.
  Then $\phi_{\#}\mu$, the pushforward of $\mu$ under $\phi$, is absolutely continuous with respect to $\mathcal L^{n}$
\end{theorem}

This leads to the following two results.
\begin{theorem}\label{thm:fractional-dimension-bound}
For $s>0$, $s\not\in\N$, let $S\subset X$ be $\Hs$ measurable with $\Hs(S)<\infty$ and $d$ the greatest integer less than $s$.
Then for every Borel $S'\subset S$ of positive measure, $\Hs \llcorner S'$ has at most $d$ independent Alberti representations.
\end{theorem}

\begin{proof}
  Let $\phi\colon X \to \R^{d+1}$ be Lipschitz and suppose that $S'\subset S$ is Borel such that $\Hs\llcorner S'$ has $d+1$ independent Alberti representations.
  Then by Theorem \ref{thm:dephilippis-marchese-rindler}, $\phi_{\#}(\Hs \llcorner S') \ll \mathcal L^{d+1}$.
  Since $\Hs(S)<\infty$ and $\phi$ is Lipschitz, $\Hs(\phi(S'))<\infty$.
  Therefore $\mathcal L^{d+1}(\phi(S'))=0$ and so $\phi_{\#}(\Hs \llcorner S')(\phi(S'))=0$.
  That is, $\Hs(S')=0$.
\end{proof}

Combining Theorem \ref{thm:dephilippis-marchese-rindler} with the Besicovitch--Federer projection theorem provides an improvement of Theorem \ref{t:alberti-rect} for subsets of Euclidean space.

\begin{theorem}\label{thm:alberti-reps-projection-thm}
Let $S\subset \R^m$ be Borel with $\Hn(S)<\infty$ such that $\Hn\llcorner S$ has $n$ independent Alberti representations.
Then $S$ is $n$-rectifiable.
\end{theorem}

\begin{proof}
  Let $\phi\colon \R^m \to \R^n$ be Lipschitz such that the Alberti representations of $\Hn\llcorner S$ are in the $\phi$-direction of independent cones $C_1,\ldots,C_n \subset \mathbb R^n$.
  Identify $S$ with its image under the biLipschitz embedding $\iota \colon x\mapsto (\phi(x),x)\in \R^n\times\R^m$ and also identify $\phi$ with the orthogonal projection onto $W := \R^n \times \{0\} \subset \R^n\times \R^m$.
By pushing forward each of the original Alberti representations, we see that $\mu:=\iota_{\#}\Hn\llcorner S$ has $n$ $\phi$-independent Alberti representations, each in the $\phi$-direction of a $C_i$.
Write the Alberti representations of $\mu$ as $(\{\mathbb P_i,\{\mu_\gamma^i\})$, for $1\leq i \leq n$.

After taking a countable decomposition of $S$, we may suppose that there exists a $\delta>0$ such that $\|\phi(\gamma'(t))\| \geq \delta \|\gamma'(t)\|$ for $\mathbb P_i$-a.e.\ $\gamma\in \Gamma(\R^n \times \R^m)$, $\Ho$-a.e.\ $t\in\dom\gamma$, and each $1\leq i \leq n$.
(This follows by applying \cite[Corollary 5.9]{Bate_2014} for each Alberti representation, with $\psi$ the orthogonal projection onto the centre of $C_i$, and slightly widening each cone such that the widened cones are also independent.)

Therefore, for any orthogonal projection $\pi$ onto an $n$-dimensional plane $V$ sufficiently close to $W$ (depending on $\delta,n,m$), $\pi_{\#}\mu$ also has $n$-independent Alberti representations.
By Theorem \ref{thm:dephilippis-marchese-rindler}, $\pi_{\#}\mu \ll \mathcal L^n$.

However, if $S$ is not $n$-rectifiable,
there exists some Borel $S'\subset S$ with $\Hn(S')>0$ that is purely $n$-unrectifiable.
Since $\Hn(S)<\infty$, we also have $\Hn(S')<\infty$.
By the Besicovitch--Federer projection theorem, there exist $V$ arbitrarily close to $W$ for which $\mathcal L^n(\pi(S'))=0$, and hence $\pi_{\#}\mu(\pi(S'))=0$.
This contradicts the fact that $\Hn(S')>0$.
\end{proof}

\begin{corollary}\label{cor:alberti-reps-projection-cor}
Let $S\subset \R^m$ be purely $n$-unrectifiable with finite $\Hn$ measure.
For any Borel $S'\subset S$ with positive $\Hn$ measure, $\Hn \llcorner S'$ has at most $n-1$ independent Alberti representations.
\end{corollary}

As noted earlier, the recent work of Csörnyei Jones allows us to remove the lower density assumption from Theorem \ref{t:alberti-rect}.
Alternatively, we may use Theorem \ref{thm:dephilippis-marchese-rindler} to remove the upper density assumption.
\begin{corollary}
  Let $(X,d,\mu)$ be a metric measure space with $\mu(X)<\infty$ and $n$ independent Alberti representations.
  Then
  \begin{equation*}
    \limsup_{r\to 0}\frac{\mu (B(x,r))}{r^{n}} < \infty \quad \text{for } \mu\text{-a.e.\ } x\in X.
  \end{equation*}
\end{corollary}
\begin{proof}
  Let $\phi\colon X \to \R^{n}$ be Lipschitz such that $\mu$ has $n$ $\phi$-independent Alberti representations.
  By Theorem \ref{thm:dephilippis-marchese-rindler}, $\phi_{\#}\mu \ll \mathcal L^{n}$ and so
  \begin{equation*}
    \limsup_{r\to 0}\frac{\phi_{\#}\mu (B(x,r))}{r^{n}} < \infty \quad \text{for } \phi_{\#}\mu\text{-a.e.\ } x\in \R^{n}.
  \end{equation*}
  In particular, since $\phi(B(x,r)) \subset B(\phi(x),\Lip \phi r)$ for each $x\in X$ and $r>0$,
  \begin{equation*}
    \limsup_{r\to 0}\frac{\mu (B(x,r))}{r^{n}} < \infty \quad \text{for } \mu\text{-a.e.\ } x\in X.
  \end{equation*}
\end{proof}

Combining Theorems \ref{thm:general-alberti-pu} and \ref{thm:fractional-dimension-bound}, Observation \ref{obs:pu-no-density} and Corollaries \ref{c:alberti-pu} and \ref{cor:alberti-reps-projection-cor} gives the following relationship between purely unrectifiable and $\tilde A$ sets.
\begin{theorem}\label{thm:atilde-summary}
  For $s>0$ let $S\subset X$ be $\Hs$ measurable with $\Hs(S)<\infty$ and let $d$ be the greatest integer strictly less than $s$.
  Suppose that either $s\not \in \N$ or $S$ is purely $s$-unrectifiable and one of the following holds:
  \begin{enumerate}
    \item \label{i:summary-1pu} $S$ is purely 1-unrectifiable (in this case, we may set $d=0$);
    \item \label{i:summary-euclidean} $X=\R^k$ for some $k\in \N$;
    \item \label{i:summary-star} $S$ satisfies \eqref{eq:lower-density}.
  \end{enumerate}
  Then for any Lipschitz $F\colon X \to \R^m$, there exists a $N\subset S$ with $\Hs(N)=0$ such that $S\setminus N \in \tilde A(F,d)$.
\end{theorem}

\begin{remark}
Note that the converse to this theorem is true for the integer case: if $S$ is not purely $n$-unrectifiable, then, if $f \colon A\subset \R^m \to S$ is biLipschitz with $\mathcal L^n(A)>0$, $S \not\in \tilde A(f^{-1},n-1)$.
\end{remark}

\begin{remark}
  \label{rmk:lower-density-unnec3}
  By using the comments in Remark \ref{rmk:lower-density-unnec}, we see that this theorem is true for all purely unrectifiable sets, without assuming \eqref{eq:lower-density}.

  The announced results of Csörnyei Jones also imply that any Lebesgue null set of $\R^n$ belongs to $\tilde A(\id,n-1)$.
  By considering projections to $n$ dimensional subspaces spanned by coordinate axes, this implies that any $N\subset \R^m$ with $\Hn(N) =0$ belongs to $\tilde A(\id,n-1)$.
  Therefore, for any $N\subset X$ with $\Hn(N)=0$ and Lipschitz $F\colon X\to \R^m$, $N\in \tilde A(F,n-1)$.
  That is, we may take $N=\emptyset$ in the previous theorem.
\end{remark}

\section{Constructing real valued perturbations} \label{s:constructing-perturbations}

First we fix some notation for this section.
\begin{notation}\label{not:1d-perturbation}
	Let $B$ be a Banach space, $T\colon B \to \R$ linear and $\delta>0$.
	Suppose that $S\subset B$ is compact and satisfies $\Ho(\gamma\cap S)=0$ for each $\gamma\in \Gamma(B)$ with
  \begin{equation}\label{eq:S-small-on-curves}(T\circ \gamma)'(t) \geq \delta \|T\|\Lip(\gamma,t) \quad \text{for }\Ho \text{-a.e.\ } t\in \dom \gamma.\end{equation}
  We let $\Omega$ be the closed convex (and hence compact) hull of $S$.
  Further, for $\gamma\in \Gamma(B)$ and $V\subset B$ Borel define
  \[R(V, \gamma, \delta)=\int_{\gamma^{-1}(B\setminus V)}(T\circ\gamma)' + \int_{\dom\gamma} \delta \|T\|\Lip(\gamma,\cdot).\]
\end{notation}

 Note that $\Ho(\gamma\cap S)=0$ for each $\gamma\in \Gamma(B)$ satisfying \eqref{eq:S-small-on-curves} is equivalent to
  \begin{equation}\label{eq:S-small-on-subset}
  \Ho(\gamma(\{t\in\dom\gamma: (T\circ\gamma)'(t) \geq \delta\|T\| \Lip(\gamma,t)\})\cap S)=0
  \end{equation}
  for all $\gamma\in \Gamma(B)$.
  Indeed, for any compact
  \[K\subset \{t\in\dom\gamma: (T\circ\gamma)'(t) \geq \delta \|T\|\Lip(\gamma,t)\},\]
  for almost every $t\in K$, $(T\circ\gamma|_K)'(t)=(T\circ\gamma)'(t)$ and $\Lip(\gamma|_K,t)=\Lip(\gamma,t)$.
  Thus $\gamma|_K$ satisfies \eqref{eq:S-small-on-curves} and so $\mathcal L^1(K)=0$ and hence \eqref{eq:S-small-on-subset}.

In this section we construct an arbitrarily small perturbation $f$ of $T$ that, when restricted to $S$, has pointwise Lipschitz constant at most $\delta$.
Suppose that $x,y\in B$ are connected by a curve $\gamma$.
By the fundamental theorem of calculus, $T(x)-T(y)=\int_{\dom\gamma}(T\circ\gamma)'$.
We will construct a function $f$ for which this integral can (almost) be replaced by $R(V, \gamma, \delta)$, for $V$ an appropriate neighbourhood of $S$ in $\Omega$.
Note that, when restricted to $S$, $f$ does have pointwise Lipschitz constant at most $\delta$, because the first integral in the definition of $R$ equals zero.

The first step is to find an appropriate $V$ such that the resulting function is a small perturbation of $T$.
Compare to \cite[Lemma 6.2]{Bate_2014}.
\begin{lemma}\label{lem:old-small-on-curves}
  For any $\eps>0$ there exists a $V\supset S$, open in $\Omega$, such that
  \begin{equation*}
    R(V, \gamma, \delta) \geq T(\gamma(l))-T(\gamma(0)) -\eps,
  \end{equation*}
  for any $l\geq 0$ and any Lipschitz $\gamma\colon [0,l] \to \Omega$ with $(T\circ \gamma)'\geq 0$ almost everywhere.
\end{lemma}

\begin{proof}
It is possible to deduce this directly from \cite[Lemma 6.2]{Bate_2014}.
However, the set up for that lemma is more technical, and also less general than the present situation.
For simplicity, we give a direct proof.

	Suppose that the conclusion is false for some $\eps>0$ and the sets
	\[V_n = \{x \in \Omega: d(x,S)<1/n\}.\]
	Then for each $n\in\N$ there exists a Lipschitz $\gamma_n \colon [0,l_n]\to \Omega$ with $(T\circ \gamma_n)'\geq 0$ almost everywhere such that
	\begin{equation}\label{eq:counter-assumption}R(V_n, \gamma_n, \delta) \leq T(\gamma_n(l))-T(\gamma_n(0)) -\eps.\end{equation}
	By the compactness of $\Omega$, we may suppose that each $\gamma_n$ has the same end points, $\gamma_s,\gamma_e\in \Omega$.
	Observe that for each $n\in \N$,
	\[\delta \|T\|\Ho(\gamma_n) \leq \int_{\dom\gamma_n} \delta \|T\|\Lip(\gamma_n,\cdot) \leq T(\gamma_e)-T(\gamma_s).\]
	Therefore, there exists an $l\geq 0$ and a reparametrisation of each $\gamma_n$ such that each is a 1-Lipschitz function defined on $[0,l]$.
	Further, by the Arzelà-Ascoli theorem and taking a subsequence if necessary, we may suppose that the $\gamma_n$ converge uniformly to some $\gamma\colon [0,l]\to \Omega$.

	Fix an $m\in \N$ and let $n\geq m$.
	Then, since $V_n\subset V_m$ and $(T\circ\gamma_n)'\geq 0$ almost everywhere,
	\[R(V_m, \gamma_n, \delta) \leq R(V_n, \gamma_n, \delta).\]
	Let $I$ be a finite collection of closed intervals contained in $\gamma^{-1}(B\setminus \overline{V_m})$, an open subset of $\R$.
	Note that both of the integrals appearing in the definition of $R(\overline{V_m}, \gamma|_
	I, \delta)$ are the total variation of Lipschitz functions.
	Thus, by the lower semi-continuity of total variation under uniform convergence,
	\begin{equation}\label{eq:first-eq}R(\overline{V_m}, \gamma|_I, \delta) \leq \liminf_{n\to\infty} R(\overline{V_m}, \gamma_n|_I, \delta).\end{equation}
	Further, since $(T\circ\gamma_n)'\geq 0$ almost everywhere and $V_n \subset \overline{V_m}$ for each $n\geq m$,
	  \begin{equation}\label{eq:second-eq}\liminf_{n\to\infty} R(\overline{V_m}, \gamma_n|_I, \delta) \leq \liminf_{n\to\infty}R(V_n, \gamma_n|_I, \delta) \leq T(\gamma(l))-T(\gamma(0)) -\eps,\end{equation}
	  where the final inequality uses \eqref{eq:counter-assumption} and the fact that $(T\circ\gamma_n)'\geq 0$ almost everywhere.
	By combining \eqref{eq:first-eq} and \eqref{eq:second-eq} and taking the supremum over all such $I$, since $\gamma^{-1}(B\setminus \overline{V_m})$ is open, we obtain
	\[R(\overline{V_m}, \gamma, \delta) \leq T(\gamma(l))-T(\gamma(0)) -\eps\]
	for each $m\in\N$.
	Since $S$ is closed, $\overline{V_m}$ monotonically decrease to $S$ as $m$ increases and so
	\begin{equation}\label{eq:third-eq} R(S, \gamma, \delta) \leq T(\gamma(l))-T(\gamma(0)) -\eps.\end{equation}

	By substituting the definition of $R$ into \eqref{eq:third-eq}, applying the fundamental theorem of calculus to the right hand side and rearranging the resulting inequality, we see that
	\[\int_{[0,l]}\delta \|T\|\Lip(\gamma,\cdot) \leq \int_{\gamma^{-1}(S)}(T\circ \gamma)' -\eps.\]
	Applying \eqref{eq:S-small-on-subset} gives
	\[\int_{[0,l]}\delta \|T\|\Lip(\gamma,\cdot) \leq \int_{\gamma^{-1}(S)}(T\circ \gamma)' -\eps \leq \int_{\gamma^{-1}(S)}\delta\|T\| \Lip(\gamma,\cdot) -\eps,\]
	which contradicts $\eps>0$.
\end{proof}

Next we extend the previous lemma to include all curves in $B$, not only those in $\Omega$.
\begin{lemma}\label{lem:whole-banach-small-on-curves}
  For any $\eps>0$, the set $V\supset S$ obtained from Lemma \ref{lem:old-small-on-curves} satisfies
  \begin{equation}
    \label{eq:growth-on-curves}
    R(V, \gamma, \delta) \geq T(\gamma(l))-T(\gamma(0)) -\eps,
  \end{equation}
  for any $l\geq 0$ and any Lipschitz $\gamma\colon [0,l] \to B$ with $(T\circ \gamma)'\geq 0$ almost everywhere.
\end{lemma}

\begin{proof}
    Fix a $\gamma$ as in the statement of the lemma.
    We will modify $\gamma$ to construct a curve in $\Omega$.

	Let $m=\min \gamma^{-1}(\Omega)$ and $M=\max \gamma^{-1}(\Omega)$.
	Since $\Omega$ is compact, $(m,M)\setminus \gamma^{-1}(\Omega)$ is open.
	Suppose that $(a,b)$ is a connected component for some $a<b$, so that $\gamma(a),\gamma(b)\in \Omega$.
	We form $\tilde \gamma$ by altering $\gamma$ in $(a,b)$ to equal the straight line segment joining $\gamma(a)$ to $\gamma(b)$.
	Since $\Omega$ is convex, this segment is contained in $\Omega$.
	Also, since $T$ is linear, we have $(T\circ \tilde\gamma)'\geq 0$ almost everywhere.
	Further, $\gamma((a,b)) \cap V = \emptyset$, $(T\circ \tilde\gamma) \leq (T\circ \gamma)$ whenever they both exist and $\Lip(\tilde\gamma,t) \leq \Lip(\gamma,t)$ for all $t$.
	Therefore
	\[R(V, \tilde\gamma|_{(a,b)}, \delta) \leq R(V, \gamma|_{(a,b)}, \delta).\]
	By repeating this for each connected component of $(m,M)\setminus \gamma^{-1}(\Omega)$, we obtain $\tilde\gamma\colon [m,M]\to \Omega$ with
	\[R(V, \tilde\gamma, \delta) \leq R(V, \gamma|_{[m,M]}, \delta)\]
	and $(T\circ \tilde\gamma)'\geq 0$ almost everywhere.
	Therefore, by applying the conclusion of Lemma \ref{lem:old-small-on-curves} to $\tilde \gamma$,
	\begin{equation}\label{eq:this-new-equation}R(V, \gamma|_{[m,M]}, \delta) \geq R(V, \tilde\gamma, \delta) \geq T(\tilde\gamma(M)) -T(\tilde \gamma(m)) -\eps.\end{equation}

	Finally, we consider the end points of $\gamma$.
	Since
	\[\gamma([0,m)) \cap \Omega = \gamma((M,l]) \cap \Omega = \emptyset,\]
	the fundamental theorem of calculus gives
	\[R(V, \gamma|_{[0,m)}, \delta) \geq T(\gamma(m))-T(\gamma(0))\]
	and
	\[R(V, \gamma|_{(M,l]}, \delta) \geq T(\gamma(l))-T(\gamma(M)).\]
	Therefore, using \eqref{eq:this-new-equation} and the fact that $\tilde\gamma(m)=\gamma(m)$ and $\tilde\gamma(M)=\gamma(M)$,
	\begin{align*} R(V, \gamma, \delta) &= R(V, \gamma|_{[0,m)}, \delta) + R(V, \gamma|_{[m,M]}, \delta) + R(V, \gamma|_{(M,l]}, \delta)\\
	&\geq T(\gamma(m))-T(\gamma(0)) + T(\tilde\gamma(M)) -T(\tilde \gamma(m)) -\eps + T(\gamma(l))-T(\gamma(M))\\
	&= T(\gamma(l))-T(\gamma(0))-\eps,
	\end{align*}
	as required.
\end{proof}

We now use the previous lemma to construct a perturbation $f$ of $T$.
This construction uses the same general idea as the one in \cite[Lemma 6.3]{Bate_2014}, but we must make adjustments to fit our current purposes.
Other than technical differences that were introduced to fit the situation in \cite{Bate_2014}, the first difference is that $f$ is defined on the whole of $B$, rather than only the compact subset $\Omega$.
This is a consequence of the previous lemma and is necessary to perturb a Lipschitz function defined on the whole of $X$, rather than simply a compact subset.

The second difference is that we now obtain a stronger Lipschitz type bound on $f$, given in \eqref{eq:basic-perturbation-lipschitz-bound}.
This is necessary for us to obtain the required bound on the Lipschitz constant of the vector valued perturbation constructed in Section \ref{sec:vect-valu-pert}.
\begin{lemma} \label{l:basic-perturbation}
 For any $\eps>0$ there exists a Lipschitz function $f \colon B \to \R$ and a $\rho >0$ such that:
 \begin{itemize}
 \item For every $y,z\in B$,
   \begin{equation}
     \label{eq:basic-perturbation-lipschitz-bound}
     |f(y)-f(z)| \leq |T(y)-T(z)| + 3\delta\|T\| d(y,z);
   \end{equation}
  \item For every $x\in B$,
    \begin{equation}
      \label{eq:basic-perturbation-close-identity}
      |T(x) - f(x)|< \eps;
    \end{equation}
  \item For every $x\in S$ and $y,z \in B(x,\rho)\cap S$,
    \begin{equation}
      \label{eq:basic-perturbation-flat}
      |f(y) - f(z)| \leq 3\delta \|T\|d(y,z).
    \end{equation}
 \end{itemize}
\end{lemma}

\begin{proof}
	First observe that if $T=0$, then $f=0$ satisfies the conclusion of the lemma.
	Therefore, we may suppose that $\|T\|\neq 0$.
	
  For $\eps>0$, let $V\supset S$ be the set, open in $\Omega$, given by Lemma \ref{lem:whole-banach-small-on-curves}.
  Define $f \colon B \to \R$ by
  \begin{equation*}
        f(x) = \inf R(V,\gamma,\delta) + T(\gamma(0)),
  \end{equation*}
  where the infimum is taken over all $l\geq 0$ and all Lipschitz $\gamma\colon [0,l] \to B$ with $(T\circ \gamma)' \geq 0$ almost everywhere and $\gamma(l)=x$.
  We call such a curve \emph{admissible} for $x$.
  For any $x\in B$, the curve consisting of the single point $x$ is admissible for $x$, and so $f(x)\leq T(x)$.
  Also, since $S$ is compact, there exists a $\rho>0$ such that $B(x,\rho)\cap\Omega \subset V$ whenever $x\in S$.
  We now show that $f$ satisfies the required conclusions for such a $\rho>0$.

  First, the fact that $f$ is Lipschitz will follow from \eqref{eq:basic-perturbation-lipschitz-bound} and the fact that $T$ is Lipschitz.
  Let $y,z\in B$ with $T(z) \geq T(y)$ and let $\gamma \colon [0,l]\to B$ be admissible for $y$.
  Define $\tilde \gamma \colon [0,l+1]\to B$ by
  \[\tilde \gamma(t) = \begin{cases}\gamma(t) & t\in [0,l]\\
      y+(t-l)(z-y) &t \in (l, l+1].\end{cases}\]
  Then $(T\circ \tilde \gamma)' \geq 0$ a.e.\ so that $\tilde \gamma$ is admissible for $z$ and so
  \begin{align}
    f(z) &\leq f(y) + R(V,\tilde\gamma|_{[l,l+1]}, \delta)\notag\\
    &\leq f(y) + \int_{[l,l+1]\setminus \tilde\gamma^{-1}(V)} (T \circ \tilde\gamma)' + \int_{[l,l+1]} \delta \|T\|\Lip(\tilde\gamma,\cdot)\notag\\
         &\leq f(y) + \delta \|T\|d(y,z) \text{ if } y,z \in B(x,\rho)\cap\Omega \subset V \label{eq:bound-in-ball}\\
    &\leq f(y) + T(z)-T(y)+ \delta\|T\| d(y,z) \text{ otherwise}. \label{eq:bound-global}
  \end{align}
  Note that \eqref{eq:bound-in-ball} uses the fact that $\tilde \gamma|_{[l,l+1]}$ is the straight line joining $y$ to $z$, which is contained in $\Omega$ and hence $V$, and \eqref{eq:bound-global} uses the fact that $(T\circ \tilde \gamma)' \geq 0$ a.e.

  To bound $f(y)$, let $\gamma \colon [0,l]\to B$ be admissible for $z$.
  If $T(\gamma(0)) \geq T(y)$ then
  \begin{equation}
  f(y) \leq T(y)
  \leq T(\gamma(0))
  \leq T(\gamma(0)) + R(V,\gamma,\delta).  \label{eq:gamma0-case}
  \end{equation}

  If $T(y) > T(\gamma(0))$, choose $v\in B$ with $T(v)=\|T\|\|v\| \neq 0$, define
  \begin{align*}
  	P\colon B &\to \ker T\\
  	x &\mapsto x-\frac{T(x)}{T(v)}v.
  \end{align*}
  and set
  \[t_{0} = \inf \{t\in [0,l] : T(\gamma(t)) \geq T(y)\}.\]
  Since $T(\gamma(0))<T(y) \leq T(z)$, we have $0<t_0\leq l$.
  
  Define $\tilde\gamma\colon [0,l+1]\to B$ by
  \[\tilde \gamma(t) = \begin{cases}\gamma(t) & t\in [0,t_{0}]\\
      \gamma(t_0) + P(\gamma(t)-\gamma(t_0)) &t \in (t_{0}, l]\\
      \tilde \gamma(l)+(t-l)(y-\tilde\gamma(l)) & t \in (l, l+1].\end{cases}\]
  Observe that $T(\tilde\gamma(t)) = T(y)$ for all $t\in [t_0,l+1]$.
  Indeed, for $t=t_0$ this follows from the definition of $t_0$; for $t\in (t_0,l]$ this is because $P$ maps into the kernel of $T$; for $t\in (l,l+1]$,
  \[T(\tilde\gamma(t)) = T(\tilde\gamma(l)) + (t-l)[T(y)-T(\tilde\gamma(l))]= T(y).\]
  In particular, $\tilde\gamma$ is an admissible curve for $y$, with $\tilde\gamma(0)=\gamma(0)$.  
  Secondly, observe that for almost every $t\in[t_{0},l]$,
  \[\Lip(\tilde\gamma,t) = \Lip(P\circ\gamma,t) \leq \Lip(\gamma,t) + \frac{\|v\|}{T(v)}(T\circ\gamma)'(t).\]
  Therefore
  \begin{align}R(V, \tilde \gamma|_{[t_{0},l]}, \delta) &= \int_{[t_0,l]\setminus \tilde\gamma^{-1}(V)} (T\circ\tilde\gamma)' + \int_{[t_0,l]} \delta\|T\| \Lip(\tilde\gamma,\cdot)\notag\\
  &\leq 0 + \int_{[t_0,l]} \delta\|T\| \Lip(\gamma,\cdot) + \int_{[t_0,l]}\frac{\delta\|T\|}{\|T\|} (T\circ \gamma)'\notag\\
   &\leq R(V, \gamma|_{[t_{0},l]},\delta) + \delta (T(\gamma(l))-T(\gamma(t_0)))\notag\\
   &\leq R(V,\gamma|_{[t_{0},l]},\delta) + \delta \|T\| d(y,z).\label{eq:use-this-align}
  \end{align}
  Further, a direct calculation shows that $\tilde \gamma(l+1)-\tilde \gamma(l)= P(z)-P(y)$ and so
  \[R(V, \tilde \gamma|_{[l,l+1]},\delta) = \delta \|T\|\|\tilde\gamma(l+1)-\tilde\gamma(l)\|\leq 2\delta\|T\| d(y,z).\]
  This and \eqref{eq:use-this-align} gives
  \begin{align}
    f(y) &\leq T(\tilde\gamma(0)) + R(V, \tilde\gamma|_{[0,t_0]},\delta) + R(V, \tilde \gamma|_{[t_0,l]},\delta) + R(V, \tilde \gamma|_{[l,l+1]},\delta)\notag\\
    &\leq T(\gamma(0)) + R(V, \gamma|_{[0,t_0]},\delta) + R(V,\gamma|_{[t_{0},l]},\delta) + \delta \|T\| d(y,z) + 2\delta\|T\| d(y,z)\notag\\
    &= T(\gamma(0)) +R(V,\gamma,\delta) + 3\delta \|T\|d(y,z).
    \label{eq:gammay-case}
  \end{align}
    By using \eqref{eq:gamma0-case} or \eqref{eq:gammay-case} depending on whether $T(\gamma(0)) \geq T(y)$ or not, and taking the infimum over all admissible $\gamma$ for $z$, we see that
    \begin{equation}
     f(y) \leq f(z) + 3\delta \|T\|d(y,z) \label{eq:y-bound}.
  \end{equation}
  Combining equations \eqref{eq:bound-in-ball}, \eqref{eq:bound-global} and \eqref{eq:y-bound} gives \eqref{eq:basic-perturbation-lipschitz-bound} and \eqref{eq:basic-perturbation-flat}.
  Note that for \eqref{eq:basic-perturbation-flat} we use the fact that $S\subset \Omega$.

  Finally, by applying \eqref{eq:growth-on-curves} to any admissible curve $\gamma$ for $x$,
  \[f(x) \geq T(x) - T(\gamma(0)) - \eps + T(\gamma(0)) = T(x)-\eps.\]
  Since $f(x)\leq T(x)$ for all $x\in B$, $f$ satisfies \eqref{eq:basic-perturbation-close-identity}.
\end{proof}

We conclude this section by describing the precise setting we will use this construction, without the fixed quantities in Notation \ref{not:1d-perturbation}.
Recall the definition of the set $E(W,\theta)$ given in Definition \ref{def:cones}.
\begin{proposition}
  \label{prop:basic-perturbation}
  Let $V$ be a finite dimensional Banach space, $F\colon X \to V$ Lipschitz, $0<\theta<1$ and $W \subset V$ a subspace.
  Suppose that a compact $S\subset X$ satisfies $\Ho(\gamma\cap S)=0$ for every $\gamma\in \Gamma(X)$ in the $F$-direction of $E(W,\theta)$.
  Further, suppose that $T\colon V\to \R$ is linear with $W\leq \ker T$.
  Then for any $\eps>0$ there exists a Lipschitz $f\colon X \to \R$ and a $\rho>0$ such that:
  \begin{itemize}
  \item For every $y,z\in X$,
    \begin{equation}\label{eq:main-construction-lipschitz}
      |f(y)-f(z)| \leq |T(F(y)-F(z))| + 3(1-\theta)\|T\| \Lip F d(y,z);
    \end{equation}
  \item For every $x\in X$,
    \begin{equation}\label{eq:prop-close}
      |T(F(x))-f(x)| < \eps;
    \end{equation}
  \item For every $x\in S$ and $y,z \in B(x,\rho) \cap S$,
    \begin{equation}\label{eq:prop-flat}
      |f(y)-f(z)| \leq 3(1-\theta) \|T\|\Lip F d(y,z).
    \end{equation}
  \end{itemize}
\end{proposition}

\begin{proof}
  First note that it suffices to prove a version of the proposition where we replace \eqref{eq:prop-close} with $|T(F(x))-f(x)| \leq \eps \|T\|\Lip F$.
  Indeed, the stated version follows from this statement since $\eps>0$ is arbitrary.
  After this modification, the hypotheses and conclusion are invariant under multiplying $F$ or $T$ by a constant.
  Therefore (since the result is true if $\Lip F$ or $\|T\|$ equals zero) we may suppose that $\Lip F =\|T\|=1$.

  Next we obtain the required Banach space structure.
  Let $\iota \colon X \to \ell_{\infty}(X)$ be an isometric embedding, for example the standard Kuratowski embedding, and let $B=V \times \ell_{\infty}(X)$.
  We equip $B$ with the norm
  \[\|(v,x)\| = \max\{\|v\|,\|x\|_\infty\}.\]
  Define $\iota^{*} \colon X \to B$ by
  \begin{equation*}
    \iota^{*}(x) = (F(x),\iota(x)).
  \end{equation*}
  Since $\Lip F = 1$, this embedding is an isometry, and so we may identify $X$ with its isometric copy in $B$.
  Moreover, on this isometric copy of $X$, $F$ agrees with the projection onto the first factor of $B$, which we also denote by $F$.
  In particular, $F$ is linear and $\Ho(\gamma \cap S) = 0$ for any $\gamma \in \Gamma(B)$ in the $F$-direction of $E(W,\theta)$.

  Now suppose that $T\colon V \to \R$ is linear with $W \leq \ker T$.
  Suppose that $\gamma \in \Gamma(B)$ with
  \[(T\circ F \circ\gamma)'(t)\geq (1-\theta)\|T\|\Lip(\gamma,t)\]
  for some $t\in \dom\gamma$.
  Then, since $W \leq \ker T$ and $\Lip F= 1$, if $w$ is a closest point of $W$ to $(F\circ\gamma)'(t)$,
\begin{align*} \|T\| d((F\circ\gamma)'(t),W) &\geq \|T((F\circ\gamma)'(t)-w)\| = |(T\circ F\circ\gamma)'(t)|\\
&\geq (1-\theta)\|T\|\Lip (\gamma,t) \geq (1-\theta) \|T\|\|(F\circ\gamma)'(t)\|.\end{align*}
  That is,
  \[d((F\circ\gamma)'(t),W) \geq (1-\theta) \|(F\circ\gamma)'(t)\|\]
  and so $(F\circ\gamma)'(t)\in E(W,\theta)$.
  Thus, since $\Ho(\gamma\cap S)=0$ for every $\gamma\in \Gamma(B)$ in the $F$-direction of $E(W,\theta)$, $\Ho(\gamma\cap S)=0$ for each $\gamma\in \Gamma(B)$ with
  \[(T\circ F \circ\gamma)(t)'\geq (1-\theta)\|T\|\Lip(\gamma,t)\]
  almost everywhere.

  Finally, to apply Lemma \ref{l:basic-perturbation} to $T\circ F$, we estimate $\|T\circ F\|$.
  Since $F$ is the projection onto the first co-ordinate,
  \[\|T\circ F\| = \sup_{\max(\|v\|,\|x\|_\infty)} T(v) = \sup_{\|v\|=1} T(v) = \|T\|.\]
  Thus, if $\gamma\in\Gamma(B)$ satisfies
  \[(T\circ F \circ\gamma)(t)'\geq (1-\theta)\|T\circ F\|\Lip(\gamma,t) \quad \text{for a.e.\ }t \in\dom\gamma,\]
  then it also satisfies
  \[(T\circ F \circ\gamma)(t)'\geq (1-\theta)\|T\|\Lip(\gamma,t) \quad \text{for a.e.\ }t \in \dom \gamma,\]
  and so $\Ho(\gamma\cap S)=0$.
  Therefore we may apply Lemma \ref{l:basic-perturbation} to $T\circ F$ with the choice $\delta=1-\theta$ to obtain a Lipschitz function $f\colon X \to \R$ and a $\rho>0$.
  The properties of $f$ we require are precisely those given by the lemma.
\end{proof}

\section{Perturbing coordinate functionals}\label{sec:vect-valu-pert}
Up to this point, the choice of norm on $\R^m$ has not been important;
The results of Section \ref{s:alberti-representations-rectifiability} are of a geometric nature concerning purely unrectifiable sets and the results of Section \ref{s:constructing-perturbations} concern real valued functions.
Throughout this section we will construct Lipschitz functions into a finite dimensional Banach space $V$, which we now fix.
We require precise control on the Lipschitz constant of these functions and so the choice of norm is very important.

In this section, all norms, Lipschitz constants and operator norms are taken with respect to $V$.

\subsection{Properties of finite dimensional Banach spaces}\label{sec:prop-finite-dimens}
We first fix some terminology for several quantitative properties of $V$.
This is required so that we may construct certain Lipschitz functions into $V$, coordinate by coordinate, in a way that the Lipschitz constant does not depend on the dimension of $V$.
For Euclidean targets, an orthonormal basis is sufficient (see Observation \ref{obs:tildek-euclidean}).
For non Euclidean targets (in particular $V=\ell_{\infty}^{m}$), something more involved is required.
The following constructions mimic the properties we require of Euclidean space in the general setting.

Fix a basis $b_1,\ldots,b_m$ of $V$ consisting of unit vectors.
We will write $b_i^{*}$ for the $i^{th}$ coordinate functional
\[b_{i}^{*}:  \sum_{i=1}^m \lambda_{i}b_{i} \mapsto \lambda_{i}.\]
By the compactness of the unit spheres in $V$ and $\ell_\infty^m$, there exists a $K_u \geq 1$ such that
\begin{equation}
  \label{eq:fin-dim-unconditional}
    \left\|\sum_{i=1}^ml_{i}b_{i}^{*}(x)b_{i}\right\| \leq K_{u} \|l\|_{\infty}\|x\|
\end{equation}
for each $x\in V$ and each $l\in \ell_\infty^m$.
Note that if $V=\ell_p^m$ for some $1\leq p \leq \infty$ and $b_1,\ldots,b_m$ is the standard basis, then $K_u=1$.

A \emph{projection} is a linear function $P$ on a vector space to itself such that $P^{2}=P$.
For an integer $d\geq 0$ and a $d$-dimensional subspace $W$ of $V$, the Kadets-Snobar theorem (\cite[Theorem 13.1.7]{Albiac_2016}) gives a projection $P\colon V \to W$ of norm at most $\sqrt{d}$.
We set $Q=\id - P\colon V \to \ker P$, so that $\|Q\|\leq \sqrt{d}+1$ and $x = P(x) + Q(x)$ for each $x\in V$.

  By applying the triangle inequality and \eqref{eq:fin-dim-unconditional}, we see that for any $x\in V$ and $l \in \ell_\infty^m$ with $\|l\|_\infty \leq 1$,
\begin{align}\label{eq:tildek-exists}
  \left\|P(x)+ \sum_{i =1}^m l_i b_i^*(Q(x))b_{i}\right\| &\leq \|P(x)\| + \left\|\sum_{i=1}^m l_i b_i^*(Q(x))b_{i}\right\|\notag \\
  &\leq \|P(x)\| + K_u\|l\|_\infty \|Q(x)\|\notag \\
  &\leq K_u (2\sqrt{d}+1) \|x\|.
\end{align}
This leads us to the following definition.
\begin{definition}
  \label{def:tilde-k}
For $V$ a finite dimensional Banach space and an integer $d\geq 0$, we let $\tilde K(V,d)$ be the least $K \geq 1$ for which the following is true:
There exists $K_d,K_p >0$ and, for any $d$-dimensional subspace $W$ of $V$, a basis $b_1,\ldots,b_m$ of $V$ consisting of unit vectors and projections $P\colon V\to W$ and $Q\colon V\to \ker P$ such that:
\begin{enumerate}
	\item \label{identity} $P(x)+ Q(x) = x$ for all $x\in V$;
  \item $\|b_i^*\|\leq K_p$ for each $1\leq i \leq m$;
  \item \label{norm-bound-pq} $\|P\|,\|Q\|\leq K_d$;
  \item For each $x\in V$ and $l\in \ell_\infty^m$ with $\|l\|_\infty \leq 1$,
\begin{equation}
  \label{eq:used-basis-constant}
  \left\|P(x)+ \sum_{i=1}^m l_i b_i^*(Q(x))b_{i}\right\| \leq K \|x\|.
\end{equation}
\end{enumerate}
\end{definition}

Note that \eqref{eq:tildek-exists} shows that $\tilde K(V,d) \leq K_u (2\sqrt{d}+1)$ for any finite dimensional Banach space $V$ and any $d\geq 0$.
We record some particular values of $\tilde K(V,d)$.
\begin{observation}\label{obs:tildek-euclidean}
For any $m\in \N$ and $d \geq 0$, $\tilde K(\ell_2^m,d)=1$.
Indeed, for any $d$-dimensional subspace $W\leq \ell_2^m$, choose an orthonormal basis $b_1,\ldots, b_m$ such that $b_{1},\ldots,b_{d}$ is a basis of $W$ and let
$P$ be the orthogonal projection onto $W$ and $Q$ the orthogonal projection onto $W^{\perp}$.
Then $K_{p}=K_{u}=K_{d}=1$ and, for any $x\in \ell_2^m$ and $l\in \ell_\infty^m$ with $\|l\|_\infty \leq 1$,
\begin{align*}\left\|P(x)+ \sum_{i=1}^m l_i b_i^*(Q(x))b_{i}\right\| &= \left\|\sum_{i=1}^d \langle x,b_i \rangle b_i + \sum_{i=d+1}^m l_i \langle x,b_i \rangle b_i\right\|\\
&\leq \|x\|,
\end{align*}
so that $\tilde K(\ell_2^m,d)=1$.
\end{observation}

\begin{observation}\label{obs:tildek-d0}
Suppose that $V$ has a basis $b_1,\ldots,b_m$ for which, in \eqref{eq:fin-dim-unconditional}, $K_u=1$.  For example, if $V=\ell_p^m$ for any $1\leq p \leq \infty$, then the standard basis satisfies this.
Then $\tilde K(V,0)=1$.
Indeed, the only $0$-dimensional subspace of $V$ is $W=\{0\}$ and so we may take $P=0$ and $Q$ to be the identity, so that $K_{d}=1$.
Then for any $l \in \ell_\infty^m$ with $\|l\|_\infty\leq 1$,
\begin{align*}\left\|P(x)+ \sum_{i=1}^m l_i b_i^*(Q(x))b_{i}\right\| &= \left\|\sum_{i=1}^m l_i b_i^*(x)b_{i}\right\|\\
&\leq \|x\|.
\end{align*}
As mentioned above, since $K_u=1$ for this basis, $\tilde K(\ell_p^m,d) \leq 2\sqrt{d}+1$ for any $d \geq 1$ and $1\leq p \leq \infty$.
Note that this is independent of $m$.
\end{observation}

To deduce one of our main theorems (Theorem \ref{thm:metric-perturbation}), we will apply the general perturbation constructed in the next subsection to the following function.
\begin{lemma}\label{lem:compact-perturbation-l_infty}
  Let $(X,d)$ be a compact metric space.
  For any $\eps>0$ there exists an $m\in\N$ and a 1-Lipschitz $F\colon X \to \ell_\infty^m$ such that
  \begin{equation*}
    \|F(x)-F(y)\|\geq d(x,y)-\eps
  \end{equation*}
  for each $x,y\in X$.
\end{lemma}

\begin{proof}
  Given $\eps>0$ let $x_1,\ldots, x_m$ be a maximal $\eps$-net of $X$ and define
  \begin{align*}
    F &\colon X \to \ell_\infty^m\\
    F(x) &= (d(x,x_1),\ldots, d(x,x_m)).
  \end{align*}
  Then $F$ is 1-Lipschitz.
  Moreover, if $x,y\in X$, there exists $1\leq i \leq m$ such that $y\in B(x_i,\eps)$.
  In particular,
  \[\|F(x)-F(y)\| \geq |d(x,x_i)-d(y,x_i)| \geq d(x,x_i) -\eps \geq d(x,y) - 2\eps.\]
  Since $\eps>0$ is arbitrary, this completes the proof.
\end{proof}

\subsection{Constructing vector valued perturbations}\label{sec:vector-constructions}
Before giving the statement and proof of the main perturbation lemma, we discuss the details of constructing Lipschitz maps $G\colon V\to V$.
We will construct such functions coordinate by coordinate.
This is because the results of Section \ref{s:constructing-perturbations} concern real valued functions.

Let $b_1,\ldots,b_m$ be a basis of $V$ consisting of unit vectors.
For 1-Lipschitz functions $f_i\colon V\to \R$, consider the map
\[G=\sum_{i=1}^m f_i b_i.\]
By simply using the triangle inequality, the estimate we obtain on the Lipschitz constant of $G$ depends on $m$.
This is not useful for our application.
Moreover, even when $m=2$, the Lipschitz constant of $G$ can be very large; consider for example the basis $(1,0),(1,\epsilon)$ of $\ell_2^2$, with $\epsilon>0$ very small, and $f_1=1$ and $f_2=0$.
Then $G$ maps $(0,\epsilon)$ to $(-1,0)$.
Thus the Lipschitz constant of $G$ is at least $1/\epsilon$.

To obtain precise control, the solution is to require
\begin{equation*}\label{componen-lipschitz}
|f_i(x)-f_i(y)|\leq |b_i^*(x-y)|
\end{equation*}
for each $x,y\in V$ and each $1\leq i \leq n$ and apply \eqref{eq:fin-dim-unconditional}.
Under this condition, maps of the form of $G$ above have Lipschitz constant at most $K_u$, and, for any $x,y\in V$ and $1\leq i \leq n$, 
\begin{equation}\label{squish-components}
|b_i^*(G(x)-G(y))| = |f_i(x)-f_i(y)|.
\end{equation}

Given a $d$-dimensional subspace $W\leq V$, we wish for a similar construction that takes into account $W$.
For example, suppose that $V=\ell_2^m$ and let $b_i$ to be an orthonormal basis where the first $d$ vectors belong to $W$.
We wish to construct a $G$ where \eqref{squish-components} holds for those $d+1\leq i \leq n$ but 
\[|b_i^*(G(x)-G(y))| = |b_i^*(x-y)| \]
for all $1\leq i \leq d$.
In this case, defining
\[G(x) = \sum_{i=1}^d b_i^*(x)b_i + \sum_{i=d+1}^m f_i(x)b_i\]
has the required properties and is 1-Lipschitz.

However, in the case $V=\ell_\infty^m$ this does not work;
Choosing a basis that is related to an arbitrary $W\leq V$ may have a very large value of $K_u$, and so such a $G$ will have large Lipschitz constant.
Therefore, we must use the standard basis of $\ell_\infty^m$.

This is where we use the definition of $\tilde K(V,d)$ given in Definition \ref{def:tilde-k}.
Precisely, by \eqref{eq:used-basis-constant}, $\tilde K(V,d)$ bounds the Lipschitz constant of the function
\[G(x) =P(x) + \sum_{i=1}^m f_i(x) b_{i},\]
whenever each $f_i\colon V\to \R$ satisfies
\begin{equation}\label{eq:needs-lipschitz}
|f_i(x)-f_i(y)|\leq |b_i^*(Q(x-y))|
\end{equation}
for each $x,y\in V$ and each $1\leq i \leq n$.

Now let $X$ be a metric space and $F\colon X \to V$ Lipschitz.
For each $1\leq i \leq n$, applying Proposition \ref{prop:basic-perturbation} with $T_i=b_i^*\circ Q$ yields a Lipschitz function $f_i\colon X \to \R$, and we define the perturbation
\[\p(x) = P(F(x)) + \sum_{i=1}^m f_i(x)b_i.\]
Note that \eqref{eq:main-construction-lipschitz} gives
\[|f_i(x)-f_i(y)| \leq |b_i^*(Q(F(x)-F(y)))|\]
(up to some error), which corresponds to \eqref{eq:needs-lipschitz}.
The important fact to note in the conclusion of the following lemma is that $\Lip\p \leq \tilde K(V,d)\Lip F$ (except for an error term that can be made arbitrarily small by choosing $\theta$ close to 1).
In particular, for the examples discussed in Observations \ref{obs:tildek-euclidean} and \ref{obs:tildek-d0}, $\tilde K(V,d)$ depends only on $d$, and for the case $\ell_2^m$ equals 1.
\begin{lemma}
  \label{lem:vector-perturbation}
  Let $F\colon X \to V$ be Lipschitz.
  Suppose that for some $d$-dimensional $W\leq V$ and $0<\theta<1$, a compact $S\subset X$ satisfies $\Ho(\gamma\cap S)=0$ for each $\gamma\in \Gamma(X)$ in the $F$-direction of $E(W,\theta)$.
  There exists a constant $C_V$ depending only upon $V$ such that the following is true.
  For any $\eps>0$ there exists $\rho>0$ and a Lipschitz $\p \colon X \to V$ such that:
  \begin{itemize}
  \item The Lipschitz constant of $\p$ is at most $(\tilde K(V,d) + (1-\theta)C_{V}))\Lip F$;
  \item For every $x\in X$,
    \begin{equation}
      \label{eq:vector-perturbation-close-identity}
      \|\p(x)-F(x)\| < \eps;
    \end{equation}
  \item For every $x\in S$ and $y,z \in B(x,\rho)$,
    \begin{equation}
      \label{eq:vector-perturbation-flat}
      \|\p(y)-\p(z)\| \leq \|P(F(y))-P(F(z))\| + (1-\theta)C_V \Lip F d(y,z)
    \end{equation}
    for $P \colon V \to W$ a projection with norm $K_{d}$.
  \end{itemize}
\end{lemma}

\begin{proof}
  Let $b_{1}, \ldots, b_{m}$ be the basis of $V$ and $P,Q$ be the two projections onto $W$ and $\ker P$ given by Definition \ref{def:tilde-k}.
  Recall that all elements of this basis have norm 1, $\|b^*_i\|\leq K_p$ for each $1\leq i \leq m$ and that $\|P\|,\|Q\|\leq K_d$.
  
  For each $1 \leq i \leq m$, set $T_i = b_i^* \circ Q$, so that each $T_i \colon V \to \R$ is linear with $\|T_{i}\| \leq K_{p}K_{d}$.
  We apply Proposition \ref{prop:basic-perturbation} to each $T_{i}$ to obtain a Lipschitz $f_{i} \colon X \to \R$ and $\rho_{i}>0$.
  We set $\rho = \min_{1\leq i \leq m}\rho_{i}>0$ and
  \[\p(x) = P(F(x)) + \sum_{i=1}^{m}f_{i}(x)b_{i}.\]
  We must establish the bound on the Lipschitz constant of $\p$ and prove equations \eqref{eq:vector-perturbation-close-identity} \eqref{eq:vector-perturbation-flat}.

  To determine the Lipschitz constant of $\p$, we use the definition of $\tilde K(V,d)$.
  Fix $y,z\in V$ and for a moment fix $1\leq i \leq m$.
  If $T_i(y)=T_i(z)$ then set $l_i=0$.
  Otherwise, let
  \[l_i = \min\left\{\frac{f_i(y)-f_i(z)}{T_i(y)-T_i(z)},1 \right\} \quad \text{if} \quad \frac{f_i(y)-f_i(z)}{T_i(y)-T_i(z)} \geq 0,\]
  and
  \[l_i = \max\left\{\frac{f_i(y)-f_i(z)}{T_i(y)-T_i(z)},-1 \right\} \quad \text{if} \quad \frac{f_i(y)-f_i(z)}{T_i(y)-T_i(z)} < 0.\]
  Then by construction $|l_i|\leq 1$ and, in any of the above three cases, equation \eqref{eq:main-construction-lipschitz} implies
  \begin{align}\label{coordinate-perturb}
    |f_i(y)-f_i(z)- l_{i}(T_i(y)-T_i(z))| &\leq (1-\theta)\|T_i\| \Lip F d(y,z) \notag\\
    &\leq (1-\theta)K_{p}K_{d} \Lip F d(y,z)\notag\\
    &= (1-\theta)C_{V} \Lip F d(y,z),
  \end{align}
  letting $C_V=K_pK_d$ in the final line.
  
  Using the triangle inequality and applying \eqref{coordinate-perturb} for each $1\leq i \leq n$ gives
  \begin{align*}
    \|\p(y)-\p(z)\| &= \left\| P(F(y)-F(z)) + \sum_{i=1}^{m} (f_{i}(y)-f_{i}(z))b_{i}\right\|\\
    &\leq \left\|P(F(y)-F(z)) + \sum_{i=1}^{m} l_{i}(T_i(y)-T_i(z))b_{i}\right\|\\
    &\qquad + \sum_{i=1}^m |f_i(y)-f_i(z) -l_i(T_i(y)-T_i(z))|\|b_i\|\\
    &\leq  \left\|P(F(y)-F(z)) + \sum_{i=1}^{m} l_{i}(T_i(y)-T_i(z))b_{i}\right\|\\
    &\qquad+ m(1-\theta)C_{V}\Lip F d(y,z).
  \end{align*}
  Substituting in for each $T_{i}$ (and absorbing the factor of $m$ into $C_{V}$) gives
  \begin{align*}
    \|\p(y)-\p(z)\| &\leq \left\|P(F(y)-F(z)) + \sum_{i=1}^m l_i b_i^*(Q(F(y) -F(z)))b_i\right\|\\
    &\qquad + (1-\theta)C_{V}\Lip F d(y,z).
  \end{align*}
  Finally, applying \eqref{eq:used-basis-constant} gives
  \begin{equation*}
    \|\p(y)-\p(z)\| \leq \tilde K(V,d)\|F(y)-F(z)\| + (1-\theta)C_{V}\Lip F d(y,z),
  \end{equation*}
  so that
  \[\Lip \sigma \leq (\tilde K(V,d) + (1-\theta)C_V) \Lip F,\]
  as required.

  The other two properties of $\sigma$ are simple consequences of the triangle inequality and the corresponding conclusions of Proposition \ref{prop:basic-perturbation}.
  Indeed, for any $x\in X$, Definition \ref{def:tilde-k} item \eqref{identity} and the definition of $b^*_i$ gives
  \begin{align*}F(x) &= P(F(x))+ Q(F(x))\\
    &= P(F(x)) + \sum_{i=1}^m b_i^*(Q(F(x)))b_i.
  \end{align*}
  Therefore, recalling that $\sigma$ is defined by
  \[\p(x) = P(F(x)) + \sum_{i=1}^{m}f_{i}(x)b_{i},\]
  and that $\|b_i\| \leq 1$ for each $1\leq i\leq n$, we have
  \begin{align*}
    \|\p(x)-F(x)\| &= \left\|\sum_{i=1}^{m}[f_{i}(x) - b_{i}^{*}(Q(F(x)))]b_{i}\right\|\\
                   &\leq \sum_{i=1}^{m} |f_{i}(x)-b_{i}^{*}(Q(F(x)))|\\
                   &= \sum_{i=1}^{m}|f_{i}(x)-T_{i}(F(x))|\\
                   &\leq m \eps,
  \end{align*}
  where the penultimate inequality simply uses the definition of $T_{i}$ and the final inequality uses \eqref{eq:prop-close}.
  Since $\eps>0$ is arbitrary, this gives \eqref{eq:vector-perturbation-close-identity}.

  Now suppose that $y,z \in B(x,\rho)$.  Since $y,z\in B(x,\rho) \subset B(x,\rho_i)$ for all $1\leq i \leq m$, the triangle inequality gives
  \begin{align*}
    \|\p(y)-\p(z))\| &\leq \|P(F(y))-P(F(z))\| + \sum_{i=1}^{m}|f_{i}(y)-f_{i}(z)|\\
    &\leq \|P(F(y))-P(F(z))\| + (1-\theta)C_{V} \Lip F d(y,z),
  \end{align*}
  where the final inequality follows by using \eqref{eq:prop-flat} for each $1\leq i \leq m$.
  This completes the proof.
\end{proof}

A general $\tilde A$ set has a finite decomposition into sets that satisfy the hypotheses of the previous lemma.
If such a set has finite measure, then up to a set of arbitrarily small measure, we may suppose that this decomposition consists of disjoint compact sets.
We will combine the corresponding perturbations we obtain from the previous lemma into a single perturbation using the following lemma.

At first thought, one may try to combine these perturbations into a single perturbation by a using a Lipschitz extension result.
However, in general, this will create a Lipschitz function with a greater Lipschitz constant than the original functions, which is not what we require.
In this lemma, the original Lipschitz function provides extra structure that enables us to maintain the same Lipschitz constant.
\begin{lemma}
  \label{lem:extension-lemma}
  Let $B$ be a normed vector space and $F\colon X \to B$ $L$-Lipschitz for some $L\geq 0$.
  Suppose that there exist $S_1,\ldots,S_M\subset X$ and $\rho_0>0$ such that the $B(S_i,\rho_0)$ are disjoint.
  Further suppose that for some $\eps>0$ there exist $L$-Lipschitz $\p_i \colon B(S_i,\rho_0)\to B$ with
  \begin{equation*}
    \|F(x)-\p_i(x)\|<\eps
  \end{equation*}
  for each $x\in B(S_i,\rho_0)$ and each $1\leq i \leq M$.
  Then there exists a $(L+2\eps/\rho_0)$-Lipschitz $\p\colon X\to V$ such that
  \begin{enumerate}
    \item \label{item:combine1}$\p(x)=\p_i(x)$ for each $x\in S_i$ and each $1\leq i \leq M$;
    \item \label{item:combine2}$\p(x)=F(x)$ if $d(x, \cup_i S_i)>\rho_0$;
    \item \label{item:combine3}$\|\p(x)-F(x)\|<\eps$ for each $x\in X$.
  \end{enumerate}
\end{lemma}

\begin{proof}
  The proof simply interpolates between the different $\p_i$.
  For each $1\leq i \leq M$ and $x\in B(S_i,\rho_0)$, write
  \begin{equation*}
    \p_i(x) = F(x) + E_i(x),
  \end{equation*}
  so that $\|E_i\|_\infty < \eps$.
  We define $\chi_i \colon X \to \R$ by
  \begin{equation*}
    \chi_i(x) = \frac{\max\{\rho_{0}/2-d(x,S_{i}),0\}}{\rho_{0}/2},
  \end{equation*}
  so that each $\chi_{i}$ equals 1 on $S_{i}$ and 0 off $B(S_{i},\rho_{0}/2)$ and so the $\chi_i$ have disjoint supports.
  Moreover, this allows us to define $\p \colon X \to B$ by
  \begin{equation*} 
    \p = F + \sum_{i=1}^M \chi_i E_i.
  \end{equation*}
  Thus properties \eqref{item:combine1}, \eqref{item:combine2} and \eqref{item:combine3} are automatically satisfied.
  It remains to check the Lipschitz constant of $\p$.

  To this end, let $y,z\in X$ and suppose that $1\leq i,j \leq M$ are such that $\chi_{i}(y)\neq 0$ and $\chi_{j}(z) \neq 0$.
  There exist at most one choice for each of $i,j$.
  If no such index exists, we choose either arbitrarily.
  First suppose that $i=j$.
  Then by the triangle inequality,
  \begin{align*}
    \|\p(y)-\p(z)\|
    &= \|F(y)-F(z) +\chi_i(y)E_i(y)-\chi_i(z)E_i(z)\|\\
    &\leq \|F(y)-F(z) + \chi_{i}(y)E_{i}(y) - \chi_{i}(y)E_{i}(z)\|\\
        &\quad + |\chi_{i}(y)-\chi_{i}(z)|\|E_{i}(z)\|\\
    &\leq \chi_i(y)\|\p_i(y)-\p_i(z)\| +(1-\chi_i(y))\|F(y)-F(z)\|\\
        & \quad + \frac{|d(y,S_{i})-d(z,S_{i})|} {\rho_{0}/2} \eps\\
    &\leq Ld(y,z) + \frac{2\eps}{\rho_0}d(y,z).
  \end{align*}

  Now suppose that $i\neq j$.
  In particular, this implies that
  \begin{equation}\label{eq:distance-apart}
    \frac{\rho_0}{2} - d(y,S_i)\leq d(y,z) \text{ and } \frac{\rho_0}{2} - d(z,S_j) \leq d(y,z).
  \end{equation}
  Indeed, suppose that the first inequality is false, then by first using the triangle inequality,
  \begin{align*}
    d(z,S_i) &\leq d(y,z) + d(y,S_i)\\
    &< \rho_0/2 -d(y,S_i) + d(y,S_i) = \rho_0/2,
  \end{align*}
  so that $\chi_i(z)\neq 0$, which contradicts any possibility of choosing $j$ as the index for $z$.
  The other inequality holds analogously.
  Thus, by the triangle inequality, \eqref{eq:distance-apart} and \eqref{item:combine3},
  \begin{align*}
    \|\p(y)-\p(z)\| &\leq \|F(y)-F(z)\| + |\chi_i(y)|\|E_i(y)\| + |\chi_j(z)|\|E_j(z)\|\\
    &\leq Ld(y,z) + 2 \frac{d(y,z)}{\rho_0}\eps.
  \end{align*}
  This establishes the required Lipschitz constant in this case.
\end{proof}

By combining the previous results, we obtain the following.

\begin{proposition}
  \label{prop:decomposition-vector-perturbation}
  Let $F\colon X \to V$ be Lipschitz.
  Suppose that for some $0<\theta< 1$, $M\in \N$ and each $1\leq i \leq M$, there exist disjoint compact sets $S_{i} \subset X$ and a $d$-dimensional $W_{i}\leq V$ such that $\Ho(\gamma\cap S_{i})=0$ for each $\gamma\in \Gamma(X)$ in the $F$-direction of $E(W_{i},\theta)$.
  There exists a $C_V\geq 1$ depending only upon $V$ such that the following is true.
  For any $\eps>0$ there exist a $\rho>0$ and a Lipschitz $\p\colon X\to V$ such that:
  \begin{itemize}
  \item The Lipschitz constant of $\p$ is at most $(\tilde K(V,d) +(1-\theta)C_V)\Lip F+\eps$;
  \item For every $x\in X$
    \begin{equation*}
      \|\p(x)-F(x)\| < \eps
    \end{equation*}
    and $\p(x)=F(x)$ if $d(x,\cup_i S_i)>\eps$;
  \item For each $1\leq i \leq M$ and $y,z\in S_{i}$ with $d(y,z)<\rho$,
    \begin{equation}
      \label{eq:decomposition-vector-perturbation-flat}
      \|\p(y)-\p(z)\| \leq \|P(F(y))-P(F(z))\| + (1-\theta)C_V \Lip F d(y,z)
    \end{equation}
    for $P_{i}\colon V \to W_{i}$ a projection with norm $K_{d}$.
  \end{itemize}
\end{proposition}

\begin{proof}
  Note that it suffices to prove the result for sufficiently small $\eps>0$ and so we fix $0<\eps<1/2$.
  Since the $S_{i}$ are a finite number of disjoint compact sets, there exists a $0<\rho_{0}<\eps$ such that the $B(S_{i},\rho_{0}$) are disjoint.
  We set $\eps'=\eps\rho_0/2<\eps$.

  For each $1\leq i \leq M$ let $\p_{i}\colon B\to V$ and $\rho_{i}>0$ be obtained by applying Lemma \ref{lem:vector-perturbation} to $S_{i}$ with the choice of $\eps'$ and let $\rho = \min_{1\leq i \leq N}\rho_{i}>0$.
  Further, we apply Lemma \ref{lem:extension-lemma} to combine these functions into a single Lipschitz function $\p\colon B \to V$.
  The conclusion of the proposition follows from the conclusions of these two lemmas, noting that combining the functions increases the Lipschitz constant by at most $2\eps'/\rho_0 <\eps$.
\end{proof}

Finally, we demonstrate how our constructed perturbation deforms the set $S$.
\begin{lemma}
  \label{lem:reduces-measure}
  Let $S\subset X$ be Borel and $F,\p\colon X \to V$ Lipschitz.
  Suppose that for some $\eps,\rho>0$ there exists a $d$-dimensional $W\leq V$ such that, for each $y,z\in S$ with $d(y,z)<\rho$,
  \begin{equation}
    \label{eq:flat-in-reduces-measure}
    \|\p(y)-\p(z)\| \leq \|P(F(y))-P(F(z))\| + \eps d(y,z),
  \end{equation}
  for $P\colon V\to W$ a projection with norm $K_{d}$.
  Then, for any real number $s>d$,
  \begin{equation*}
  \Hs(\p(S)) \leq \eps^{s-d}C_{d,s,V,F}\Hs(S),
  \end{equation*}
  for $C_{d,s,V,F}$ a constant depending only upon $d$, $s$, $V$ and $\Lip F$.
\end{lemma}

\begin{proof}
  Note that, if $\Hs(S) = \infty$, then there is nothing to prove and so we may suppose that $\Hs(S)<\infty$.
  For any $0<\delta<\rho$ we cover $S$ by sets $S_{i}$, of diameter at most $\delta$ such that
  \begin{equation}\label{eq:Si-approximate-Hausdorff}
    \sum_{i\in\N}(\diam S_{i})^{s} \leq \Hs(S) +\delta.
  \end{equation}

  We will use the $\p(S_{i})$ to create a finer covering of $\p(S)$.
  To this end, fix $i\in \N$.
  Then $P(F(S_{i}))\subset W$ is a set of diameter $\Lip P \Lip F \diam S_{i}$ contained in a $d$-dimensional subspace of $V$.
  Therefore, it may be covered by $M=C_{V,d} \eps^{-d}$ sets $T_{1},\ldots,T_{M}$ of diameter
  \[\eps \Lip P \Lip F \diam S_{i}.\]
  (Indeed, this is true if $V$ were Euclidean, and $V$ is $C_{V}$-biLipschitz to Euclidean space.)
  For each $1\leq j \leq M$, \eqref{eq:flat-in-reduces-measure} gives
  \begin{align}
    \diam \p((P\circ F)^{-1}(T_{j})\cap S_{i}) &\leq \diam T_{j} + \eps \diam (P\circ F)^{-1}(T_{j})\cap S_{i}\notag\\
                                               &\leq \eps \Lip P \Lip F \diam S_{i} + \eps\diam S_{i}\label{eq:small-diam}\\
                                               &\leq (\Lip P \Lip F +1)\eps \delta\label{eq:applicable-for-hausdorff}.
  \end{align}
  Since
  \begin{equation*}
    \p(S_{i}) = \bigcup_{j=1}^{M} \p((P\circ F)^{-1}(T_{j}) \cap S_{i}),
  \end{equation*}
  if we set $\delta' = \delta\eps(\Lip P \Lip F+1)$, then \eqref{eq:applicable-for-hausdorff} shows that this decomposition may be used to bound $\mathcal H^s_{\delta'}$.
  Using \eqref{eq:small-diam} and the fact that $M=C_{V,d} \eps^{-d}$, this gives
  \begin{align*}
    \mathcal H^{s}_{\delta'}(\p(S_{i})) &\leq \sum_{j=1}^M (\eps \diam S_i (\Lip P\Lip F +1))^s\\
    &\leq C_{V,d} \eps^{-d} (\eps\diam S_{i}(\Lip P \Lip F +1))^{s}\\
    &=C_{V,d} \eps^{s-d}(\Lip P \Lip F +1)^{s}(\diam S_{i})^{s}.
  \end{align*}
  Thus, by \eqref{eq:Si-approximate-Hausdorff},
  \begin{align*}
    \mathcal H^{s}_{\delta'}(\p(S)) &\leq C_{V,d}\sum_{i\in\N} \eps^{s-d}(\Lip P \Lip F +1)^{s}(\diam S_{i})^{s}\\
    &\leq C_{V,d} \eps^{s-d} (\Lip P \Lip F +1)^{s}(\Hs(S)+\delta).
  \end{align*}

  Since $\delta>0$ and hence $\delta'>0$ is arbitrary, we obtain \[\Hs(\p(S))\leq C_{V,d} \eps^{s-d}(\Lip P\Lip F+1)^{s}\Hs(S).\]
  Recall that $\|P\| \leq K_{d}$, so that the constant has the required form.
\end{proof}

To conclude, we summarise the results of this section.
Recall the notion of an $\tilde A$ set given in Definition \ref{def:atilde}.
\begin{theorem}\label{thm:general-perturbation}
  Let $V$ be a finite dimensional Banach space and $F\colon X \to V$ Lipschitz.
  For an integer $d \geq 0$ and a real number $s>d$, let $S \in \tilde A(F,d)$ have finite $\Hs$ measure.

  Then for any $\eps>0$ there exists a $(\tilde K(V,d) \Lip F+\eps)$-Lipschitz $\p\colon X \to V$ such that
  \begin{enumerate}
  \item \label{i:main-thm-perturbation} $\|\p(x)-F(x)\|<\eps$ for each $x\in X$ and $\p(x)=F(x)$ whenever $d(x,S)>\eps$;
  \item \label{i:main-thm-measure} $\Hs(\p(S))<\eps$.
  \end{enumerate}
\end{theorem}

\begin{proof}
  We will prove the Theorem for an arbitrary $0<\eps'<1$, which we now fix.
  Choose $0<\theta <1$ sufficiently close to 1 and $0<\eps<\eps'/2$ sufficiently small such that $(1-\theta)C_V\Lip F+\eps <\eps'$, where $C_{V}$ is the constant appearing in Proposition \ref{prop:decomposition-vector-perturbation}.
  We will impose further constraints on the size of $\eps>0$ (depending only upon $d$, $s$, $V$ and $F$) at the end of the proof.
  Note that, if $m\leq d$, then the result is immediate.
  Indeed, because $s>d \geq m$ we have $\Hs(F(S))=0$ and so choosing $\p=F$ suffices.
  Otherwise, by the definition of an $\tilde A(F,d)$ set, there exists a disjoint Borel decomposition $S = S_{1}\cup\ldots\cup S_M$ and $d$-dimensional subspaces $W_{i}\leq V$ such that each $S_i$ satisfies $\Ho(\gamma \cap S_i)=0$ for each $\gamma \in \Gamma(X)$ in the $F$-direction of $E(W_i,\theta)$.
  We also fix $\eta>0$ to be chosen at the end of the proof (in a way depending only upon $d$, $s$, $V$ and $F$).
  Then, since $\Hs(S)<\infty$, there exist compact $S_i'\subset S_i$ such that $\Hs(S\setminus \cup_i S'_{i})<\eta$.
  Note that we also have $\Ho(\gamma \cap S'_i)=0$ for each $\gamma \in \Gamma(X)$ in the $F$-direction of $E(W_i,\theta)$ for each $1\leq i \leq M$.

  We now have all of the requirements to apply Proposition \ref{prop:decomposition-vector-perturbation} to $\cup_i S'_{i}$ and $F$.
  This gives a $\rho>0$ and a Lipschitz $\p\colon X\to V$ such that:
  \begin{enumerate}
  \item \label{item:mainproof1} The Lipschitz constant of $\p$ is at most
  \[(\tilde K(V,d) +(1-\theta)C_V)\Lip F+\eps \leq \tilde K(V,d) \Lip F + \eps';\]
  \item \label{item:mainproof2} For every $x\in X$, $\|\p(x)-F(x)\|< \eps$ and $\p(x)=F(x)$ if $d(x,\cup_i S'_i)>\eps$ and hence if $d(x,S)>\eps$;
  \item \label{item:mainproof3} For each $1\leq i \leq M$, and $y,z\in S_{i}$ with $d(y,z)<\rho$,
    \begin{align*}
      \|\p(y)-\p(z)\| &\leq \|P(F(y))-P(F(z))\| + (1-\theta)C_V \Lip F d(y,z)\\
      &\leq \|P(F(y))-P(F(z))\| + \eps d(y,z),
    \end{align*}
    for $P_{i}\colon V\to W_{i}$ a projection with norm $K_{d}$.
  \end{enumerate}

  Points \eqref{item:mainproof1} and \eqref{item:mainproof2} now allow us to deduce all of the required properties of the theorem except for bounding the measure of the image, which we deduce from \eqref{item:mainproof3} and Lemma \ref{lem:reduces-measure}.
  Indeed, \eqref{item:mainproof3} is precisely the hypotheses required to apply Lemma \ref{lem:reduces-measure} to each $S'_{i}$, and so we deduce that
  \begin{equation*}
    \Hs(\p(S'_{i})) \leq \eps^{s-d} C_{d,s,V, F}\Hs(S'_{i})
  \end{equation*}
  for each $1\leq i \leq M$.
  Therefore,
  \begin{align*}
    \Hs(\p(S)) &\leq \Hs \left(\p\left(S\setminus \bigcup_{i=1}^{M} S'_{i}\right)\right) + \sum_{i=1}^{M}\Hs(\p(S'_{i}))\\
               &\leq (\Lip \p)^s \Hs \left(S\setminus \bigcup_{i=1}^{M} S'_{i}\right) + \eps^{s-d}C_{d,s,V,F}\sum_{i=1}^{M}\Hs(S'_{i})\\
    &\leq \eta (\Lip \sigma)^s + \eps^{s-d}C_{d,s,V,F}\Hs(S).
  \end{align*}
  Since $s>d$, we may choose $\eps,\eta$ sufficiently small such that this quantity is less than $\eps'$.
\end{proof}

\section{Typical Lipschitz functions}\label{sec:typical}
In this section we will consider typical Lipschitz functions defined on a metric space, equipped with the topology of uniform convergence.
Precisely, we will consider the following spaces.
\begin{definition}
  \label{def:lipschitz-space}
  For a metric space $Y$, let $\Lip(X,Y)$ be the vector space of all bounded Lipschitz functions $f\colon X\to Y$ equipped with the supremum norm.
  Note that, even if $Y$ is complete, $\Lip(X,Y)$ is not.
  However, for $L \geq 0$ the closed subspace $\Lip(X,Y,L)$ consisting of all $L$-Lipschitz $f\in\Lip(X,Y)$ is a complete metric space whenever $Y$ is complete.
  For example, this is true whenever $Y$ is a finite dimensional Banach space.

  Note that the space $\Lipo$ discussed in the introduction is $\Lip (X,\ell_2^m,1)$.
\end{definition}

A subset $R$ of a metric space $Y$ is \emph{residual} if it contains a countable intersection of open dense sets.
Recall that the Baire Category Theorem states that a residual subset of a complete metric space is dense.
Also, by definition, residual sets are closed under taking countable intersections and supersets.
Thus, residual sets form a suitable notion of ``generic points'' in a complete metric space.
When dealing with a set of continuous functions with the supremum norm, it is common to say that a certain property is \emph{typical} if the set of functions with the property is a residual set.

If a finite dimensional Banach space $V$ and an integer $d$ are chosen so that $\tilde K(V,d)=1$, then the results from the previous section perturb any element of $\Lip(X,V,L)$ into a function that is \emph{almost} in $\Lip(X,V,L)$,
the only problem being the arbitrarily small increase in the Lipschitz constant.
This can easily be corrected with the following simple scaling argument.
\begin{lemma}
  \label{lem:remove-small-increase-in-lipschitz}
  Let $V$ be a normed vector space and $L > 0$.
  For any $\eps>0$ and $f\in \Lip(X,V,L)$, there exist a $\delta>0$ and a $g \in \Lip(X,V,L-\delta)$ such that $\|f-g\|<\eps$.
\end{lemma}

\begin{proof}
  For any $\eps>0$ and $f\in \Lip(X,V,L)$, let $\delta = \eps/2L\|f\|$ and set $g=(L-\delta)f/L$ (if $\|f\|=0$ then the result is immediate).
  Then $g \in \Lip(X,V,L-\delta)$ and, for any $x\in X$,
  \begin{equation*}
    \|f(x)-g(x)\| = \left(1- \frac{L-\delta}{L}\right)\|f(x)\| = \frac{\delta}{L}\|f(x)\| < \eps,
  \end{equation*}
  as required.
\end{proof}

The results of the previous section establish the density of certain subsets of $\Lip(X,V,L)$.
We now show that these set are open, so that we may form residual sets.
\begin{lemma}
  \label{lem:small-covering-is-open}
  Let $X,Y$ be metric spaces, $L\geq 0$ and $\eps,s>0$.
  Suppose that $S\subset X$ is compact.
  The set of all $f\in \Lip(X,Y,L)$ for which $f(S)$ may be covered by open balls
  \begin{equation*}
    f(S) \subset \bigcup_{i\in\N} B(c_{i},r_{i})
  \end{equation*}
  with $\sum_i r_i^s<\eps$, is open.
\end{lemma}

\begin{proof}
  Let $f \in \Lip(X,Y,L)$ such that
  \[f(S) \subset \bigcup_{i\in\N} B(c_{i},r_{i}),\]
  for open balls $B(c_i,r_i)$ with $\sum_i r_i^s<\eps$.
  Since $S$ and hence $f(S)$ is compact, there exists a $\rho>0$ such that the $\rho$-neighbourhood of $f(S)$ is also contained in $\cup_{i} B(c_{i},r_{i})$.
  In particular, if $g\in B(f,\rho)$,
  \begin{equation*}
    g(S) \subset \bigcup_{i\in\N} B(c_{i},r_{i}).
  \end{equation*}
  Thus, the set of all such $f$ is open, as required.
\end{proof}

By a suitable countable decomposition into sets of the form in the previous lemma, we obtain the following.
\begin{theorem}
  \label{thm:general-residual}
  For $s>0$ let $S\subset X$ be $\Hs$-measurable with $\sigma$-finite $\Hs$ measure and let $d\in \N$ with $d<s$ and $L\geq 0$.
  Also, let $V$ be a finite dimensional Banach space with $\tilde K(V,d)=1$.
  Suppose that, for any Lipschitz $f\colon X \to V$, there exists an $N\subset S$ with $\Hs(N)=0$ such that $S\setminus N \in \tilde A(f,d)$.
  Then the set
  \[\{f\in \Lip(X,V,L): \Hs(f(S))=0\}\]
  is residual in $\Lip(X,V,L)$.
\end{theorem}

\begin{proof}
  Note that if $L=0$ then there is nothing to prove and so we may suppose that $L>0$.
  We first prove the result under the additional assumption that $S$ is compact and has finite $\Hs$ measure.
  Under this assumption, for any $\eps>0$, Lemma \ref{lem:small-covering-is-open} shows that the set $R_\eps(S)$ of all $f\in \Lip(X,V,L)$ for which $f(S)$ may be covered by open balls
  \begin{equation*}
    f(S) \subset \bigcup_{i\in\N} B(c_{i},r_{i})
  \end{equation*}
  with $\sum_i r_i^s<\eps$, is open.

  To see that $R_\eps$ is dense, let $f\in \Lip(X,V,L)$ and let $S'$ be the full measure subset of $S$ that belongs to $\tilde A (f,d)$.
  Since $\tilde K(V,d)=1$, for any $\eps>0$, by combining Theorem \ref{thm:general-perturbation} and Lemma \ref{lem:remove-small-increase-in-lipschitz}, there exists a $\p \in \Lip(X,V,L)$ with $\|f-\p\|<\eps$ and $\Hs(\p(S'))<\eps$.
  Indeed, given $r>0$ we apply Lemma \ref{lem:remove-small-increase-in-lipschitz} to get a $\delta>0$ and a $g\in \Lip(X,V,L-\delta)$ with $\|f-g\|<r/2$.
  We then apply Theorem \ref{thm:general-perturbation} to $g$ with the choice $\eps=\min\{\eps, r/2, \delta\}$ to get a $\p \in \Lip(X,V,L)$ with $\|\p-g\|<r/2$ and $\Hs(\p(S'))<\eps$.
  Since $\p$ is Lipschitz, $\Hs(\p(S\setminus S'))=0$, so that $\Hs(\p(S))<\eps$ and hence $\p \in R_{\eps}$.
  In particular, $\p\in R_{\eps}$ and $\|\p - f\|<r$.
  Since $r>0$ is arbitrary, $R_{\eps}$ is dense.

  By combining these two facts, each $R_{\eps}$ is residual and hence so is
  \[R(S):=\bigcap_{i\in\N}R_{1/n}(S).\]
  If $f\in R_{1/n}(S)$ then
  \[\Hs_{1/n^{1/s}}(f(S)) \leq 1/n\]
  and so $\Hs(f(S))=0$ for any $f\in R(S)$.
  This proves the theorem for this special case.

  Now suppose that $S$ is simply $\Hs$-measurable with $\sigma$-finite $\Hs$ measure.
  Then by the inner regularity of measure, there is a decomposition
  \[S=N\cup \bigcup_{i\in\N}S_i\]
  where $\Hs(N)=0$ and each $S_i$ is compact with $\Hs(S_i)<\infty$.
  Since each $S_i$ is a subset of $S$, the hypothesis on $S$ is also true for each $S_i$.
  Thus, by the previous part of the proof, we know that each $R(S_i)$ is residual and hence so is
  \[R^*:=\bigcap_{i\in\N}R(S_i).\]

  If $f\in R^*$ then $\Hs(f(S_i))=0$ for each $i\in\N$ and so $\Hs(f(\cup_i S_i))=0$ too.
  Moreover, since $\Hs(N)=0$, we have $\Hs(f(N))=0$ for any $f\in \Lip(X,V,L)$.
  Therefore, $\Hs(f(S))=0$ for any $f\in R^*$.
\end{proof}

\section{Typical Lipschitz images of purely unrectifiable sets}\label{sec:proof-main-theorems}

We begin with the first theorem stated in the introduction.
Recall the definition of $\Lip(X,V,L)$ from Definition \ref{def:lipschitz-space}.
\begin{theorem}\label{thm:residual-euclidean-target}
  For $n\in \N$ suppose that $S\subset X$ is purely $n$-unrectifiable and has a countable measurable decomposition $S=\cup_i S_i$ where each $S_i$ satisfies \eqref{eq:lower-density} and $\Hn(S_i)<\infty$.
  Then for any $L \geq 0$ and any $m\in \N$ the set
  \[\{f\in \Lip(X,\ell_2^m,L): \Hn(f(X))=0\}\]
  is residual in $\Lip(X,\ell_2^m,L)$.
\end{theorem}

\begin{proof}
  By applying Theorem \ref{thm:atilde-summary} \eqref{i:summary-star} with $s=n$, we see that, for any Lipschitz $f\colon X \to \R^m$, there exists an $N\subset S$ with $\Hn(N)=0$ such that $S\setminus N \in \tilde A(f,n-1)$.
  By Observation \ref{obs:tildek-euclidean}, we know that $\tilde K(\ell_2^m,n)=1$. 
  Thus all of the hypotheses of Theorem \ref{thm:general-residual} are satisfied and its conclusion agrees with the conclusion of the theorem.
\end{proof}

When the purely unrectifiable set is a subset of some Euclidean space, we may use Theorem \ref{thm:atilde-summary} \eqref{i:summary-euclidean} and so do not need to assume \eqref{eq:lower-density}.
In fact, because the hypothesis and conclusion of \ref{thm:atilde-summary} \eqref{i:summary-euclidean} are invariant under re-norming, this holds in any finite dimensional Banach space.
\begin{theorem}\label{thm:residual-euclidean-domain}
	Let $V$ be a finite dimensional Banach space and, for $n\in \N$, let $S\subset V$ be purely $n$-unrectifiable and have $\sigma$-finite $\Hn$ measure.
  Then for any $L>0$ and $m\in\N$ the set
  \[\{f\in \Lip(V,\ell_2^m,L): \Hn(f(S))=0\}\]
  is residual in $\Lip(V,\ell_2^m,L)$.
\end{theorem}

By using the $s\not\in \N$ case in Theorem \ref{thm:atilde-summary}, we prove the result for fractional dimension sets.
\begin{theorem}\label{thm:fractional-dimension}
  For $s\not\in\N$ let $S\subset X$ be $\Hs$-measurable with $\sigma$-finite $\Hs$ measure.
  Then for any $L \geq 0$ and any $m\in \N$ the set
  \[\{f\in \Lip(X,\ell_2^m,L): \Hs(f(X))=0\}\]
  is residual in $\Lip(X,\ell_2^m,L)$.
\end{theorem}

If the set is purely 1-unrectifiable, then we prove our results without assuming \eqref{eq:lower-density} and also for many more targets.
\begin{theorem}\label{thm:residual-domain-1pu}
  For $s>0$ let $S\subset X$ have $\sigma$-finite $\Hs$ measure.
  Suppose that either $s \in \N$ and $S$ is purely 1-unrectifiable or $0<s<1$.
  Then for any $1\leq p \leq \infty$, $m\in \N$ and any $L \geq 0$, the set
  \[\{f\in \Lip(X,\ell_p^m,L): \Hs(f(S))=0\}\]
  is residual in $\Lip(X,\ell_p^m,L)$.
\end{theorem}

\begin{proof}
  Depending on the value of $s$, we either use the $s\not\in \N$ case in Theorem \ref{thm:atilde-summary} or Theorem \ref{thm:atilde-summary} \eqref{i:summary-1pu} to deduce that, for any Lipschitz $f\colon X \to \R^m$, there exists a $N\subset S$ with $\Hs(N)=0$ such that $S\setminus N \in \tilde A(f,0)$.
  Recall from Observation \ref{obs:tildek-d0} that $\tilde K(\ell_p^m,0)=1$ for any $1\leq p\leq \infty$ and any $m\in \N$.
  Thus all of the hypotheses of Theorem \ref{thm:general-residual} are satisfied and its conclusion agrees with the conclusion of the theorem.
\end{proof}

We now turn out attention to perturbing distances in a compact metric space using functions with controlled Lipschitz constant.
\begin{theorem}\label{thm:metric-perturbation}
  For $s>0$ let $X$ be a compact metric space with $\Hs(X)<\infty$.
  Suppose that either $s\in\N$, $X$ is purely $s$-unrectifiable and satisfies \eqref{eq:lower-density} or $s\not\in\N$.
  Then for any $\eps>0$ there exists an $m\in\N$ and a $(2\sqrt{s}+1)$-Lipschitz $\p\colon X \to \ell_\infty^m$ such that
  \begin{itemize}
    \item $|d(x,y)-\|\p(x)-\p(y)\|_\infty|<\eps$ for each $x,y\in X$ and
    \item $\Hs(\p(X))<\eps$.
  \end{itemize}
\end{theorem}

\begin{proof}
  Fix $\eps>0$.
  Since $X$ is compact, we apply Lemma \ref{lem:compact-perturbation-l_infty} to obtain an $m\in\N$ and a 1-Lipschitz function $F\colon X \to \ell_\infty^m$ such that
\begin{equation}\label{eq:4}|d(x,y)-\|F(x)-F(y)\||<\eps\end{equation}
  for each $x,y\in X$.
  By Theorem \ref{thm:atilde-summary}, there exists a $N\subset S$ with $\Hs(N)=0$ such that $S\setminus N \in \tilde A(F,d)$, for $d$ the greatest integer strictly less than $s$.
  Applying Theorem \ref{thm:general-perturbation} to $F$ gives a $\p\colon X \to \ell_\infty^m$ such that
\begin{equation}\label{eq:5}|F(z)-\p(z)|<\eps\end{equation}
  for each $z\in X$ and $\Hs(\p(S))<\eps$.
  Note that, by Observation \ref{obs:tildek-d0}, $\p$ is $2\sqrt{d}+1$ Lipschitz.

  Using \eqref{eq:4}, \eqref{eq:5} and the triangle inequality gives
  \begin{align*}|d(x,y)-\|\p(x)-\p(y)\|| &\leq |d(x,y)-\|F(x)-F(y)\||\\
  &\qquad + |\|F(x)-F(y)\|-\|\p(x)-\p(y)\||\\
                                         &\leq \eps + \|F(x)-F(y)-(\p(x)-\p(y))\|\\
                                         &\leq 3\eps
  \end{align*}
  for each $x,y\in X$.
  Since $\eps>0$ is arbitrary, this completes the proof.
  \end{proof}
\begin{remark}\label{rmk:other-theorems-are-better}
  Note that, if $X$ is a subset of some Euclidean space, is purely 1-unrectifiable or $0<s < 1$, then a much stronger conclusion is obtained from Theorem \ref{thm:residual-euclidean-domain} or Theorem \ref{thm:residual-domain-1pu}.
  One simply needs to choose a Lipschitz function arbitrarily close to the identity in the first case, or a Lipschitz function arbitrarily close to the function obtained from Lemma \ref{lem:compact-perturbation-l_infty} for the latter two.
  In either case, this perturbation can be chosen to be 1-Lipschitz.
\end{remark}

\begin{remark}
  \label{rmk:lower-density-unnec2}
  If the reader accepts the first statement in Remark \ref{rmk:lower-density-unnec3}, then the lower density assumption \eqref{eq:lower-density} is not necessary in any of the previous theorems.
\end{remark}

\subsection{Perturbing sets in unconditional Banach spaces}\label{sec:pert-sets-uncond}

The concepts discussed in Section \ref{sec:prop-finite-dimens} may be generalised to infinite dimensional Banach spaces, as can be found in any introductory book on the geometry of Banach spaces, for example \cite{Albiac_2016}.
A \emph{Schauder basis} of a Banach space $X$ is a sequence $b_{j}\in X$ such that any $x\in X$ has a unique representation $x = \sum_j \lambda_{j}b_{j}$.
A well known application of the Banach-Steinhaus theorem is that the \emph{basis projections}
\[P_{n}:  \sum_{j\in \N} \lambda_{j}b_{j} \mapsto  \sum_{j = 1}^{n} \lambda_{j}b_{j}\]
are uniformly bounded (\cite[Proposition 1.1.4]{Albiac_2016}).
This leads to the \emph{bounded approximation property} for Banach spaces with a Schauder basis:
for any compact $S\subset X$ and any $\eps>0$ there exists an $m\in\N$ such that $\|P_m(x)-x\|<\eps$ for each $x\in S$.
Thus, any compact subset of $X$ may be $\eps$-perturbed into a finite dimensional subspace $V_{n}:= \sspan\{b_{1},\ldots b_{n}\}$ using a Lipschitz (in fact linear) function whose Lipschitz constant is independent of $\eps$.

We will apply Theorem \ref{thm:general-perturbation} to the $P_{m}$.
For this to be useful, we must consider the values of $\tilde K(V_{m},d)$.
A Schauder basis is \emph{unconditional} if for every $x\in X$ the sum
$\sum_jb_{j}^{*}(x)b_{j}$
converges unconditionally (i.e. independently of the order of summation).
It follows (\cite[Proposition 3.1.3]{Albiac_2016}) that there exists a constant $K_{u}$ such that, for any bounded sequence $l=(l_{i})$ and $x\in X$,
\[
  \left\|\sum_{i\in\N}l_{i}b_{i}^{*}(x)b_{i}\right\| \leq K_{u} \|l\|_{\infty}\|x\|.
\]
Therefore, for any $m\in\N$, $V_m$ satisfies \eqref{eq:fin-dim-unconditional} for this value of $K_u$.
Consequently, $\tilde K(V_m,d)$ is uniformly bounded in $m$ for each $d\geq 0$.
We denote this bound by $\tilde K(X,d)$.

Therefore, we can prove the following.
\begin{theorem}\label{thm:banach-space-perturbation}
  Let $X$ be a Banach space with an unconditional basis and, for $s>0$, let $S\subset X$ be compact with $\Hs(S)<\infty$.
  Suppose that either $s\in\N$, $S$ is purely $s$-unrectifiable and satisfies \eqref{eq:lower-density} or $s\not\in\N$.
  Then for any $\eps>0$ there exists a Lipschitz $\p\colon X \to X$ such that
  \begin{itemize}
    \item $\|\p(x)-x\|<\eps$ for each $x\in S$ and
    \item $\Hs(\p(S))<\eps$.
  \end{itemize}
  The Lipschitz constant of $\p$ depends only upon $X$ and $s$. 
\end{theorem}

\begin{proof}
  Let $M>0$ be a uniform bound for the basis projections $P_{m}$ and, for $\eps>0$, let $m\in\N$ such that $\|P_m(x)-x\|<\eps$ for each $x\in S$.
  By applying Theorem \ref{thm:atilde-summary}, there exists a $N\subset S$ with $\Hs(N)=0$ such that $S\setminus N\in \tilde A(P_m,d)$, for $d$ the greatest integer strictly less than $s$.
  By Theorem \ref{thm:general-perturbation}, there exists a $\tilde K(X,d) M$-Lipschitz $\p\colon X \to V_m$ such that $\Hs(\p(X))<\eps$ and $\|\p(x)-P_m(x)\|<\eps$ for each $x\in S$.
  Thus, the triangle inequality concludes the proof.
\end{proof}

In certain situations this can be improved.
\begin{theorem}
	 Let $X=\ell_p$ for some $1\leq p< \infty$, or $X=c_0$, and for $s>0$ let $S\subset X$ be $\Hs$ measurable with $\sigma$-finite $\Hs$ measure.
	 Suppose that either
	 \begin{itemize}
	 	\item $S$ is purely 1-unrectifiable;
	 	\item $X=\ell_2$ and $s\not\in \N$;
	 	\item $X=\ell_2$, $S$ is purely $s$-unrectifiable and has a countable measurable decomposition $S=\cup_i S_i$ where each $S_i$ satisfies \eqref{eq:lower-density} and $\Hs(S_i)<\infty$.
	 \end{itemize}
	 Then for any $\eps>0$ there exists a 1-Lipschitz $\p\colon X \to X$ such that
  \begin{itemize}
    \item $\|\p(x)-x\|<\eps$ for each $x\in S$ and
    \item $\Hs(\p(S))=0$.
  \end{itemize}
\end{theorem}

  \begin{proof}
  In this case, $V_m=\ell_p^m$ or $V_m=\ell_\infty^m$ for each $m\in\N$ and $P_m$ is the projection to the first $m$ standard basis vectors, so that $\Lip P_m=1$.
  If $X=\ell_2$, we use Theorem \ref{thm:residual-euclidean-target} or Theorem \ref{thm:fractional-dimension} to find a $\p\in \Lip(X,V_m,1)$ arbitrarily close to $P_m$ with $\Hs(\p(X))=0$.
  If $S$ is purely 1-unrectifiable then we use Theorem \ref{thm:residual-domain-1pu} instead.
  \end{proof}

\section{Typical Lipschitz images of rectifiable sets}\label{sec:converse}

We now show that a typical image of an $n$-rectifiable metric space (of positive measure) has positive $\mathcal{H}^{n}$ measure:
\begin{theorem}
  \label{thm:positive-measure-is-open-dense}
  Let $S\subset X$ be $n$-rectifiable with $\Hn(S)>0$.
  For any finite dimensional Banach space $V$ with $\dim V \geq n$ and $L >0$, the set
  \[\{f\in\Lip(X,V,L) : \Hn(f(S))>0\},\]
  is open and dense.
\end{theorem}

The most fundamental results regarding rectifiable metric spaces are due to Kirchheim.
Specifically, we will make use of \cite[Lemma 4]{kirchheim-regularity}, which we paraphrase as follows.
\begin{lemma}
  \label{lem:kirchheim}
  Let $E\subset \R^{n}$ be a Borel set and $h\colon E \to X$ Lipschitz.
  There exists a countable number of Borel sets $E_i\subset E$ such that:
  \begin{itemize}
    \item $\Hn(h(E)\setminus \cup_i h(E_i))=0$;
    \item $h$ is biLipschitz on each $E_i$.
  \end{itemize}
  In particular, for any $n$-rectifiable $S\subset X$, there exists a countable number of \emph{biLipschitz} $h_i\colon A_i \to S$ with $\Hn(S\setminus \cup_i h_i(A_i))=0$.
\end{lemma}

\subsection{The set is open}\label{sec:s-is-open}
Our preliminary results will concern arbitrary metric space targets.
This will allow us to also prove the converse to Theorem \ref{thm:metric-perturbation}.

By the result of Kirchheim above and the Vitali covering theorem, any $n$-rectifiable metric space is, up to a set of measure zero, given by a countable disjoint union of biLipschitz images of subsets of balls in $\R^n$.
Each of these subsets may be chosen to have arbitrarily large Lebesgue density in each of their respective ball.
In this subsection, we will prove results about perturbations of such high density subsets of balls, and use them to deduce that the set of Theorem \ref{thm:positive-measure-is-open-dense} is open.

We begin with a topological observation.
For this subsection we fix $n\in \N$ and let $\B$ be the unit ball of $\R^{n}$.
\begin{lemma}
  \label{lem:degree}
  Let $f\colon \B \to \B$ continuous.
  For some $0<\eps<1/2$ suppose that $\|f(x)-x\|<\eps$ for each $x\in\partial \B$.
  Then $f(\B)\supset B(0,1-\eps)$.
\end{lemma}

\begin{proof}
  There are many ways to prove this lemma.
  We give a proof that does not rely on the constructions of algebraic topology, only Brouwer's fixed point theorem.

  First let $P \colon \B \to \B$ be defined by
  \begin{align*}
    P(\lambda v) =
    \begin{cases}
      v & \lambda \in [1-\eps,1]\\
      \frac{\lambda}{1-\eps}v & \lambda \in [0,1-\eps]
    \end{cases}
  \end{align*}
  whenever $v\in \partial \B$ and $\lambda \in [0,1]$.
  Then $P\circ f\colon \B \to \B$ is continuous and maps $\partial \B$ to $\partial \B$.
  Moreover,
  \begin{equation*}
    \|P(f(x))-x\| \leq \|P(f(x))-f(x)\| + \|f(x)-x\| \leq 2\eps
  \end{equation*}
  for each $x\in \partial \B$.

  Suppose that $x\in B(0,1-\eps) \setminus f(\B)$.
  Since $P$ is bijective on $B(0,1-\eps)$, $P(x)\in B(0,1) \setminus P(f(\B))$.
  Let $\rho\colon \B \to \partial \B$ be the radial projection from $P(x)$.
  Then $F=\rho\circ P \circ f \colon \B \to \partial \B$ is continuous with $\|F(x)-x\|<2\eps<1$ for each $x\in \partial \B$.
  In particular, $-F(x)\neq x$ for each $x\in \B$.
  Thus $-F$ is a continuous function from $\B$ to itself without a fixed point, contradicting Brouwer's fixed point theorem.
\end{proof}

We obtain the following consequence for metric space targets.

\begin{lemma}
  \label{lem:converse-defined-on-ball}
  For any $L,K >0$ there exists an $\eps>0$ such that the following is true.
  For any metric space $(Y,\rho)$ and any continuous $f\colon \mathbb B \to Y$ with
  \begin{equation}\label{hyp1}
   \|x-y\|/K -\eps \leq \rho(f(x),f(y))\leq L\|x-y\|
  \end{equation}
  for each $x,y\in \partial \mathbb B$,
  \begin{equation*}
    \mathcal H^{d}(f(\mathbb B))\geq 1/2K\sqrt{n}.
  \end{equation*}
\end{lemma}

\begin{proof}
  We simply construct a Lipschitz function $P_\eps$ that maps $f(\B)$ back to $\R^m$ in such a way that the hypotheses of the previous lemma are satisfied.
  By controlling the Lipschitz constant of $P_\eps$, this gives a lower bound to the measure of $f(\B)$.

  To this end, for $\eps>0$ to be determined later, let $N$ be a maximal $\eps$-net in $\partial \mathbb B$ and let $f$ satisfy \eqref{hyp1} with the choice $\eps = \eps^2/K$.
  Then, for any $s,t\in N$,
  \begin{equation*}
    \rho(f(s),f(t)) \geq \|s-t\|/K -\eps^2/K \geq \|s-t\|(1-\eps)/K
  \end{equation*}
  and so
  \begin{equation*}
    f^{-1}|_{f(N)}
  \end{equation*}
  is $K/(1-\eps)$-Lipschitz.
  Therefore, it may be extended to a $K\sqrt{n}/(1-\eps)$-Lipschitz function $P'\colon f(\mathbb B) \to \R^{n}$.
  Observe that $P_{\eps}:=P'\circ f$ fixes $N$ and is $KL\sqrt{n}/(1-\eps)$-Lipschitz on $\partial \mathbb B$, so that
  \begin{equation*}
    |P_{\eps}(x)-x|< \eps(1+KL\sqrt{n}/(1-\eps)):=\eps^{*}
  \end{equation*}
  for every $x\in \partial \mathbb B$.

  By Lemma \ref{lem:degree},
  \begin{equation*}
    P_{\eps}(\B) \supset B(0,1-\eps^{*})
  \end{equation*}
  whenever $\eps$ is sufficiently small such that $0<\eps^{*}<1/2$.
  Therefore,
  \begin{equation*}
  \Hn(P_{\eps}(\mathbb B)) \geq (1-2\eps^{*})^{n}.
  \end{equation*}
  However, $P_{\eps}(\mathbb B) = P'(f(\mathbb B))$ and so, since $P'$ is $K\sqrt{n}/(1-\eps)$-Lipschitz,
  \begin{equation*}
    \Hn(f(\mathbb B))\geq (1-2\eps^{*})^{n}(1-\eps)/K\sqrt{n} \geq 1/2K\sqrt{n},
  \end{equation*}
  provided we reduce $\eps>0$ further if necessary.
\end{proof}

By a suitable Lipschitz extension, we may remove the topological assumptions on the domain.

\begin{lemma}
  \label{lem:converse-dense-subset-of-ball}
  For any $L,K >0$ there exists an $\eps>0$ such that the following is true.
  For any metric space $(Y,\rho)$, any Borel $E\subset \mathbb B$ with $\mathcal L^{n}(E)\geq (1-\eps)\mathcal L^{n}(\mathbb B)$ and any $L$-Lipschitz $f\colon E \to Y$ with
  \begin{equation*}
    \rho(f(x),f(y)) \geq \|x-y\|/K -\eps
  \end{equation*}
  for each $x,y\in E$,
  \begin{equation*}
    \Hn(f(E))\geq 1/4K\sqrt{n}.
  \end{equation*}
\end{lemma}

\begin{proof}
  The lemma follows by simply extending any function defined on a $E\subset \mathbb B$ to the whole of $\mathbb B$ and observing that, if $E$ has sufficiently large measure, the hypotheses of the previous lemma apply.

  To this end, suppose that $\delta,\eps>0$ and $E\subset \mathbb B$ satisfies $\mathcal L^{n}(E)\geq (1-\delta)\mathcal L^{n}(\mathbb B)$, $(Y,\rho)$ is a metric space and $f\colon E\to Y$ is $L$-Lipschitz with
  \begin{equation*}
    \rho(f(x),f(y)) \geq \|x-y\|/K -\eps
  \end{equation*}
  for each $x,y\in E$.

  Since $f(E)\subset Y$ is separable, we may isometrically embed $f(E)$ into $\ell_{\infty}$ and extend it, component by component, to an $L$-Lipschitz function $f\colon \mathbb B\to Y':=f(\mathbb B)\subset \ell_{\infty}$.
  If $0<\delta<(\eps/2)^{n}$, then we have $B(E,\eps)\supset \mathbb B$ and so, given $x,y\in\mathbb B$, there exists $x',y'\in E$ with $\|x-x'\|,\|y-y'\|<\eps$.
  In particular,
  \begin{align*}
    \|f(x)-f(y)\|_{\infty} &\geq \|f(x')-f(y')\|_{\infty}-\|f(x)-f(x')\|_{\infty}-\|f(y)-f(y')\|_{\infty}\\
                           &\geq \|x'-y'\|/K - 2L\eps\\
                           &\geq \|x-y\|/K -2(L+1/K)\eps.
  \end{align*}

  Now suppose that $\eps_{1}>0$ is given by the previous lemma and $\eps>0$ is sufficiently small such that $2(L+1/K)\eps\leq\eps_{1}$.
  Then we may apply the previous lemma to $f$ to see that $\Hn(f(\mathbb B))\geq 1/2K\sqrt{n}$.
  However, since $f$ is $L$-Lipschitz,
  \begin{equation*}
    \Hd(f(\mathbb B\setminus E)) \leq L^{n} \Hn(\mathbb B\setminus E) \leq L^{n}\delta.
  \end{equation*}
  In particular, provided $\delta \leq 1/4KL^{n}\sqrt{n}$, we have $\Hn(f(E))\geq 1/4K\sqrt{n}$, as required.
  Thus, choosing $\eps=\min\{1/4KL^{n}\sqrt{n},\eps_{1}/2(L+1/K)\}$ is sufficient.
\end{proof}

Finally, by scaling, we may apply the previous result to a ball of any radius.
\begin{lemma}
  \label{lem:converse-dense-subset-any-radius}
  For any $L,K >0$ there exists an $\eps>0$ such that the following is true.
  For any metric space $(Y,\rho)$, any $x\in \R^n$, $r>0$ and $E\subset B(x,r)$ with $\mathcal L^n(E) \geq (1-\eps)\mathcal L^n(B(x,r))$ and any $L$-Lipschitz $f\colon E \to Y$ with
  \begin{equation*}
    \rho(f(x),f(y)) \geq \|x-y\|/K -\eps r
  \end{equation*}
  for each $x,y\in E$,
  \begin{equation*}
    \Hn(f(E))\geq r^n/4K\sqrt{n}.
  \end{equation*}
\end{lemma}

\begin{proof}
  This simply follows from the previous lemma by a scaling argument.
  Consider the scaled metric space $Y_r := (Y,\rho/r)$ and the function
  \begin{align*}
    G &\colon B(0,1) \to Y_r\\
    G(y) &= f((y-x)/r),
  \end{align*}
  Then, because of the choice of metric in $Y_r$, $G$ is $L$-Lipschitz and
  \begin{equation*}
   \frac{\rho(G(x),G(y))}{r} = \frac{\rho(f(x),f(y))}{r} \geq \frac{\|x-y\|}{r K} -\eps.
 \end{equation*}
 Moreover, the scaled copy $(E-x)/r$ of $E$ inside $B(0,1)$ satisfies $\mathcal L^{n}((E-x)/r)\geq (1-\eps)\mathcal L^{n}(B(0,1))$.
 Therefore, we may apply the previous lemma to conclude that $\Hn(G((E-x/r)))\geq 1/4K\sqrt{n}$ with respect to the metric $\rho/r$.
 That is, $\Hn(f(E))\geq r^{n}/4K\sqrt{n}$.
\end{proof}

We are now in the position to prove the converse direction to Theorem \ref{thm:intro-main-theorem-2}.
Following this, we will use it to prove that the set from Theorem \ref{thm:positive-measure-is-open-dense} is open.
\begin{theorem}\label{thm:metric-converse}
  Let $S\subset X$ be $n$-rectifiable with $\Hn(S)>0$.
  Then,
  \begin{equation*}
    \inf_{L\geq 1}\lim_{\eps\to 0} \inf\Hn(\p_\eps(S)) > 0
  \end{equation*}
  where the second infimum is taken over all metric spaces $(Y,\rho)$ and all $L$-Lipschitz $\p_\eps\colon X \to Y$ with $|d(x,y)-\rho(\p_\eps(x),\p_\eps(y))|<\eps$ for each $x,y\in S$.
\end{theorem}

\begin{proof}
  By applying Lemma \ref{lem:kirchheim}, there exists a Borel $E\subset \R^n$ of positive measure and $K$-biLipschitz $h\colon E \to S$, for some $K\geq 1$.
  Observe that, if $\eps>0$, $(Y,\rho)$ is a metric space and $\p\colon X \to Y$ is $L$-Lipschitz with $|d(x,y)-\rho(\p(x),\p(y))|<\eps$ for each $x,y\in S$, then
  \begin{equation}
    \label{eq:composition-perturbation}
    \|x-y\|/K -\eps \leq \rho(\p(h(x)),\p(h(y))) \leq KL\|x-y\|
  \end{equation}
  for each $x,y\in E$.

  Fix $L\geq 1$ and let $\eps_{2}>0$ be given by the previous lemma for the choice of $KL$ in place of $L$.
  By applying the Vitali covering theorem, there exists a finite collection of disjoint closed balls $B_{i}\subset \R^{n}$ such that
  \begin{equation}
    \label{eq:balls-cover-E}
    \mathcal L^{n}\left(E \setminus \bigcup_{i=1}^{M}B_{i}\right)<\mathcal L^{n}(E)/2
  \end{equation}
  and
  \begin{equation}\label{eq:E-dense-in-balls}
    \mathcal L^{n}(E\cap B_{i})\geq \max\{(1-\eps_{2}),1/2\}\mathcal L^{n}(B_{i})
  \end{equation}
 for each $i\in\N$.
  Since the $B_{i}$ are a finite number of disjoint closed balls, there exists an $\eps_{0}>0$ such that $B(B_{i},\eps_{0})\cap B(B_{j},\eps_{0})=\emptyset$ whenever $i\neq j$.
  For each $1\leq i \leq M$ let $r_{i}$ be the radius of the $B_{i}$ and $r=\min r_{i}$.
  We now fix $0<\eps<\min\{r\eps_{2},\eps_{0}/K\}$, a metric space $(Y,\rho)$ and a $L$-Lipschitz $\p\colon X \to Y$ with $|d(x,y)-\rho(\p(x),\p(y))|<\eps$ for each $x,y\in X$.

  Note that, since the $B_{i}$ are separated by a distance at least $\eps_{0}>K\eps$, equation \eqref{eq:composition-perturbation} shows that the $\p(h(B_{i}))$ are disjoint.
  Therefore, by applying Lemma \ref{lem:converse-dense-subset-any-radius}
    to each $B_{i}$ for $1\leq i \leq M$, we see that
 \begin{equation*}
   \Hn(\p(h(E\cap \bigcup_{i=1}^{M}B_{i}))) \geq \sum_{i=1}^{M}\frac{r_{i}^{n}}{4K\sqrt{n}} = \frac{\mathcal L^{n}(\cup_{i}B_{i})}{\mathcal L^{n}(\mathbb B) 4K\sqrt{n}}\geq \frac{\mathcal L^{n}(E)}{\mathcal L^{n}(\mathbb B)16K\sqrt{n}}.
 \end{equation*}
 Note that the final inequality uses equations \eqref{eq:balls-cover-E} and \eqref{eq:E-dense-in-balls}.
 Since $\p(X)\supset \p(h(E))$, and the right hand side of this expression involves quantities depending only on $E$, this completes the proof.
\end{proof}

As a consequence, we now prove that the set in Theorem \ref{thm:positive-measure-is-open-dense} is open.
In fact, we prove the following stronger result.
\begin{proposition}
  Let $S\subset X$ be $n$-rectifiable.
  For any $L> 0$, any metric space $Y$ and any $L$-Lipschitz $f\colon X \to Y$ with $\Hn(f(S))>0$, there exists an $\eps>0$ such that $\Hn(g(S))>0$ for any $L$-Lipschitz $g \colon S \to Y$ with $\rho(f(x),g(x))<\eps$ for each $x\in S$.
\end{proposition}

\begin{proof}
  Fix a metric space $(Y,\rho)$ and Lipschitz $f\colon X \to Y$.
  Note that, if $f$ were injective, then any perturbation of $f$ would also induce a perturbation of $f(S)$, so that the previous theorem can be applied.
  If $f$ were biLipschitz, then any Lipschitz perturbation of $f$ would introduce a Lipschitz perturbation of $f(S)$.
  We will prove the Proposition by reducing to this case.

  By Lemma \ref{lem:kirchheim}, there exists a countable number of biLipschitz $h_i\colon A_i\subset \R^n \to S$ with $\Hn(S\setminus \cup_i h_i(A_i))=0$.
  Since $\Hn(f(S))>0$, there exists some $A_i$ with $\Hn(f(h_i(A_i)))>0$.
  Moreover, by applying Lemma \ref{lem:kirchheim} to $f\circ h_i$, there exists some $A\subset A_i$ of positive measure on which $f\circ h_i$ is biLipschitz.
  In particular, $f$ is $M$-biLipschitz on $h_i(A)$ for some $M\geq 1$.
  Let $Y'=f(h_i(A))$.

  Now fix $L >0$.
  By Theorem \ref{thm:metric-converse}, there exists an $\eps>0$ such that $\Hn(\p(Y'))>0$ for each $LM$-Lipschitz $\p\colon Y' \to Y$ with $|\rho(x,y)-\rho(\p(x),\p(y))|<\eps$ for each $x,y\in Y'$.
  Notice that, if $g\colon S \to Y$ is $L$-Lipschitz with
\begin{equation}\label{eq:g-close-to-f}\rho(f(x),g(x))<\eps/2 \quad \text{for each } x\in S,\end{equation}
  then $\p := g\circ f^{-1} \colon Y' \to Y$ is $LM$-Lipschitz and
  \begin{align*}|\rho(w,z)-\rho(\p(w),\p(z))| &= |\rho(f(f^{-1}(w)),f(f^{-1}(z))) - \rho(g(f^{-1}(w)), g(f^{-1}(z)))| \\
                                              \begin{split}&\leq
                                                |\rho(f(f^{-1}(w)),f(f^{-1}(z))) - \rho(f(f^{-1}(z)), g(f^{-1}(w)))|\\ &\ +  |\rho(f(f^{-1}(z)),g(f^{-1}(w))) - \rho(g(f^{-1}(w)), g(f^{-1}(z)))|
                                              \end{split}\\
                                              &\leq \rho(f(f^{-1}(w)), g(f^{-1}(w))) + \rho(f(f^{-1}(z)),g(f^{-1}(z)))\\
                                              &\leq \eps/2 + \eps/2=\eps,
  \end{align*}
  using the reverse triangle inequality for the penultimate inequality and \eqref{eq:g-close-to-f} for the final inequality.
  Therefore, we may apply the conclusion of the previous theorem to $\p$ to see that $\Hn(\p(Y'))>0$.
  Since $\p(Y')=g(f^{-1}(Y'))=g(h_i(A)) \subset g(S)$, we have $\Hn(g(S))>0$, as required.

\end{proof}

\subsection{The set is dense}

We now prove that the set in Theorem \ref{thm:positive-measure-is-open-dense} is dense.
The main step is to prove that we can perturb any Lipschitz function between two Euclidean spaces to have positive measure image.
Recall that the set of invertible linear functions is a dense open subset of all linear functions $\R^n \to \R^n$.
Moreover, $T\mapsto \|T^{-1}\|$ is continuous on this set.
The main step follows naturally by modifying a Lipschitz function around a point of differentiability, in such a way that the derivative of the modified function is invertible.
This leads to the required result.
\begin{lemma}
  \label{lem:basic-perturb-positive-image}
  Let $A\subset \R^n$ be a Borel set with positive measure, $m\geq n$ and $f\colon A \to \R^m$ Lipschitz.
  For any $\eps>0$ there exists a Lipschitz $T^*\colon \R^n \to \R^m$ with
  \[\Lip T^*,\ \|T^*\|_\infty<\eps\]
  such that
  \begin{equation*}
    f^* := f + T^*
  \end{equation*}
  has $\Hn(f^*(A)) >0$.
\end{lemma}

\begin{proof}
  Since $f$ is Lipschitz, its derivative $Df(x)$ exists for almost every $x\in A$.
  Moreover, standard measure theoretic techniques show that $Df$ is a Borel function defined on a full measure Borel subset of $A$.
  Thus, there exists a $A'\subset A$ of positive measure on which $Df$ is continuous.
  Further, standard techniques also show that, for any $\eps>0$, the function $R_\eps(x)$ defined to be the greatest $R$ such that
  \begin{equation}
    \label{eq:derivative-control}
    \|f(y)-f(x) -Df(x)(y-x)\|<\eps \|y-x\| \quad \forall y\in B(x,R)
  \end{equation}
  is also Borel.
  Thus, there exists $A'' \subset A'$ of positive measure and, for every $\eps>0$, a $R_\eps>0$ such that $R_\eps < R(x)$ for each $x\in A''$.
  We let $x_0$ be a density point of $A''$.

  Since $m\geq n$, there exists an $n$-dimensional subspace $W\leq \R^m$ that contains the image of $Df(x_0)$.
  Given $\eps>0$, there exists an invertible linear $S\colon \R^n \to W$ with $\|Df(x_0) - S\|<\eps$.
  Moreover, there exists a $\delta>0$ and $M\in \N$ such that $\|L^{-1}\|\leq M$ for each $L\in B(S,\delta)$.
  We let $T = S - Df(x_0)$ and $\tilde f=f+T$.
  Note that $\Lip T<\eps$.

  Since $Df$ is continuous on $A''$, there exists an $R_*>0$ such that $\|Df(x)-Df(x_0)\|<\delta$ whenever $\|x-x_0\|<R_*$.
  In particular, this implies that
  \begin{equation*}
    \|S-(T+Df(x))\| = \| Df(x) - Df(x_0)\| <\delta,
  \end{equation*}
  so that $T+Df(x)$ is invertible with $\|(T+Df(x))^{-1}\| \leq M$.
  That is,
  \begin{equation}
    \label{eq:norm-inverse-control}
    \|y-x\| \leq M \| (Df(x)+T)(y-x) \| \quad \forall y\in \R^n,\ x\in A'' \cap B(x_0, R_*).
  \end{equation}
  Moreover, if $x\in A''$ and $\|y-x\|<R_{1/2M}$, then by \eqref{eq:derivative-control},
  \begin{align*}
    \| \tilde f(y)-\tilde f(x) - (Df(x)+T)(y-x) \| &= \|f(y)-f(x) - Df(x)(y-x) \| \\
                                         &\leq \|y-x\|/2M.
  \end{align*}
  Thus, by the reverse triangle inequality and \eqref{eq:norm-inverse-control},
  \begin{equation*}
    \|\tilde f(y)-\tilde f(x)\| \geq \|y-x\|/2M
  \end{equation*}
  whenever $y\in \R^n$ and $x\in A'' \cap B(x_0, R)$ for $R=\min\{R_*, R_{1/2m}\}$.
  That is, $\tilde f$ is biLipschitz on $A''\cap B(x_0,R)$.
  Since $x_0$ is a density point of $A''$, $A''\cap B(x_0,r)$ has positive measure for each $0<r<R$ and hence so does $\tilde f(A'' \cap B(x_0,r))$.

  Finally, we define
  \[
  T^*(x)=\begin{cases}T(x-x_0) & \text{if } \|x-x_0\|\leq 1 \\
  \frac{T(x-x_0)}{\|x-x_0\|} & \text{otherwise.}\end{cases}
  \]
  Then $\Lip T^* \leq \Lip T <\eps$, $\|T^*(x)\| < \eps$ for all $x\in \R^n$ and $T^*=T$ on $B(x_0,1)$.
  Thus, if $f^*=f+T^*$, $\Hn(f^*(A))>0$, as required.
\end{proof}

To apply this in the metric case, we apply the results of Kirchheim.
\begin{proposition}
  \label{prop:expanding-measure-is-dense}
  Let $S\subset X$ be $n$-rectifiable with $\Hn(S)>0$.
  Suppose that $V$ is a finite dimensional Banach space with $\dim V \geq n$ and $L >0$.
  Then for any $f\in \Lip(X,V,L)$ and any $\eps>0$, there exists a $g\in \Lip(X,V,L)$ with $\|f-g\|<\eps$ such that $\Hn(g(S))>0$.
\end{proposition}

\begin{proof}
  First note that it suffices to prove the result for $V=\R^m$ for some $m \geq n$, since the result is invariant under biLipschitz mappings of $V$.
  This allows us to apply the previous lemma.

  By Lemma \ref{lem:remove-small-increase-in-lipschitz}, there exists a $\delta>0$ and a $\tilde f \in \Lip(X,\R^m,L-\delta)$ with $\|f-\tilde f\|<\eps/2$.
  By Lemma \ref{lem:kirchheim}, there exists a biLipschitz $h\colon A\subset \R^n \to S$ with $\mathcal L^n(A)>0$.
  We extend $h^{-1}$ to a Lipschitz function $h^{-1} \colon X \to \R^n$.
  Finally, by applying Lemma \ref{lem:basic-perturb-positive-image} to $\tilde f\circ h^{-1} \colon A \to \R^m$, we see that there is a Lipschitz $T^* \colon \R^n \to V$ with $\Lip T^* <\delta/\Lip h^{-1}$ and $\|T^*\|_\infty <\eps/2$ such that $f^*:= \tilde f + T^*$ has $\Hn(f^*(A))>0$.

  We claim that $g:=f^*\circ h^{-1}$ is the required function.
  Certainly $g(S)\supset f^*(A)$, so that $\Hn(g(S))>0$.
  Also note that for any $x\in X$,
  \begin{equation*}
    \|g(x)-f(x)\| \leq \|g(x)-\tilde f(x)\| + \|\tilde f(x) -f(x)\| < \|T^*(x)\| + \eps/2 \leq \eps.
  \end{equation*}
  Therefore, $\|g-f\| <\eps$.
  Finally, $\Lip f^* \leq \Lip f + \Lip T^* \leq L -\delta + \delta$, so that $g\in \Lip(X,\R^m,L)$, as required.
\end{proof}

The previous proposition completes the proof of Theorem \ref{thm:positive-measure-is-open-dense}.

We may also deduce the following topological consequence of our perturbation results.
Note that in Euclidean space, this can be deduced using the Besicovitch-Federer projection theorem in place of our perturbation theorem.
Recall that $\B$ is the unit ball of $\R^{n}$.
\begin{theorem}
  \label{thm:topological-consequences}
  Let $f\colon \mathbb B\to X$ continuous and biLipschitz on $\partial \mathbb B$.
  Suppose that there exists a countable Borel decomposition $f(\mathbb B)= \cup_i X_i$ such that each $X_i$ satisfies \eqref{eq:lower-density} and $\Hn(X_i)<\infty$.
  Then $f(\mathbb B)$ contains an $n$-rectifiable subset of positive $\Hn$ measure.
  That is, $f(\mathbb B)$ is not purely $n$-unrectifiable.

  If $n=1$ then this is true for any $f(\B)$ with $\sigma$-finite $\Ho$ measure.
\end{theorem}

\begin{proof}
  Consider $g:= f^{-1}|_{f(\partial \mathbb B)} : \partial \mathbb B \to \partial \mathbb B$.
  This is a Lipschitz function and so may be extended to a Lipschitz function $g \colon f(\mathbb B) \to \R^n$.
  Since $f(\mathbb B)$ is compact, $g$ is bounded and so $g\in \Lip (f(\mathbb B), \R^n,L)$ for some $L >0$.

  Suppose that $f(\mathbb B)$ is purely $n$-unrectifiable.
  Since each $X_i$ satisfies \eqref{eq:lower-density} and $\Hn(X_i)<\infty$, we may apply Theorem \ref{thm:residual-euclidean-target} to get a $h \in \Lip(f(\mathbb B), \R^n,L)$ with $\|g-h\|<1/4$ and $\mathcal L^n(h(f(\mathbb B)))=0$.
  In particular, $h(f(\mathbb B)) \not\supset B(0,1/10)$.
  However, for any $x\in \partial \mathbb B$,
  \begin{align*}
    \|h(f(x))-x\| &\leq \|h(f(x)) - g(f(x))\| + \|g(f(x))- x\|\\
    &\leq \|h-g\| + \|f^{-1}(f(x))-x\| < 1/4.
  \end{align*}
  Thus, we obtain a contradiction of Lemma \ref{lem:degree}.

  If $n=1$, then we may apply Theorem \ref{thm:residual-domain-1pu} instead of Theorem \ref{thm:residual-euclidean-target} to deduce the same conclusion without assuming each $X_i$ satisfies \eqref{eq:lower-density}.
\end{proof}

\begin{remark}
  As previously, by using the contents of Remark \ref{rmk:lower-density-unnec2}, we may remove the lower density assumption \eqref{eq:lower-density} from the hypotheses of the previous theorem.
\end{remark}

\bibliographystyle{abbrvurl}
\bibliography{references}
\end{document}